\documentclass[12pt]{amsart}
\usepackage{amsmath,amssymb,latexsym, amsthm, amscd, mathrsfs, stmaryrd, mathtools}
\usepackage[linktocpage=true]{hyperref}
\usepackage{cleveref}
\usepackage{xcolor}
\usepackage[all]{xy}
\usepackage{enumitem}
\usepackage{chngcntr}
\usepackage{multirow}
\usepackage{array}
\usepackage{booktabs}
\usepackage{tikz}
\usepackage[sorting=nyt,style=alphabetic,backend=bibtex,hyperref=true,doi=false,maxbibnames=9,maxcitenames=4,url=false]{biblatex}
\crefformat{section}{\S#2#1#3} 
\crefformat{subsection}{\S#2#1#3}
\crefformat{subsubsection}{\S#2#1#3}

\newtheorem{thm}{Theorem} [section]
\theoremstyle{definition}

\theoremstyle{plain}
\newtheorem{prop}[thm]{Proposition}

\newtheorem{lem}[thm]{Lemma}
\newtheorem{cor}[thm]{Corollary}

\numberwithin{equation}{section}

\newcommand{\Hom}{\mathrm{Hom}}
\newcommand{\Ext}{\mathrm{Ext}}

\newcommand{\C}{\mathbb C}

\newcommand{\g}{\mathfrak{g}}
\newcommand{\bo}{\mathfrak{b}}
\newcommand{\gl}{\mathfrak{gl}}
\newcommand{\h}{\mathfrak{h}}

\newcommand{\n}{\mathfrak n}

\newcommand{\one}{{\ov 1}}
\newcommand{\osp}{\mathfrak{osp}}
\newcommand{\ov}{\overline}
\newcommand{\OO}{\mathcal O}

\newcommand{\W}{\mathcal W}

\newcommand{\Z}{\mathbb Z}
\newcommand{\oo}{{\ov 0}}

\makeatletter
\newcommand{\extp}{\@ifnextchar^\@extp{\@extp^{\,}}}
\def\@extp^#1{\mathop{\bigwedge\nolimits^{\!#1}}}
\makeatother

\newcommand{\red}[1]{{\color{red}#1}}

\title[Characters for Projective Modules in the BGG Category $\OO$]{Characters for projective modules in the BGG category $\OO$ for the orthosymplectic Lie superalgebra $\mathfrak{osp}(3|4)$}

\author[A.S. Kannan]{Arun S. Kannan}
\address{Department of Mathematics, Massachusetts Institute of Technology, Cambridge, MA 02139} \email{akannan@mit.edu}

\author[H. Zhu]{Honglin Zhu}
\address{Phillips Exeter Academy, 20 Main Street, Exeter, NH 03833} \email{hzhu2@exeter.edu}

\begin{document}

\begin{abstract}
We determine the Verma multiplicities of standard filtrations of projective modules in integral atypical blocks in the BGG category $\OO$ for the orthosymplectic Lie superalgebras $\osp(3|4)$ by way of translation functors. We then explicitly determine the composition factor multiplicities of Verma modules using BGG reciprocity.
\end{abstract}

\maketitle

\setcounter{tocdepth}{1}
\tableofcontents

\section{Introduction}

\subsection{} A central problem in representation theory is understanding the representations of a given algebraic object, like a semisimple Lie algebra, and in particular, determining the irreducible constituents. A class of representations in which this problem is accessible is the BGG category $\OO$ of modules of semisimple Lie algebras. This category exhibits rich and deep theory and a broad survey of results can be found in \cite{humphreys2008representations}. A generalization of semisimple Lie algebras is basic Lie superalgebras, which exhibit many of the same phenomena (for reference, see \cite{cheng2012dualities, musson2012lie}). The BGG category $\OO$ can analogously be defined for basic Lie superalgebras, and many of the results from the semisimple case extend. Among the most conceptual objects in this category are the Verma modules. In this paper, we determine Verma multiplicities of standard filtrations of projective modules of integral atypical highest weight in the BGG category $\OO$ for the basic Lie superalgebra $\osp(3|4)$. We then use BGG reciprocity to determine the composition factors in Verma modules.

\subsection{}
Atypicality of weights is a phenomenon present in Lie superalgebras that has no analogue for semisimple Lie algebras. It allows an integral block in $\OO$ whose degree of atypicality is greater than $0$ to have infinitely many simple modules. The principal block in $\OO$ for $\osp(2m+1|2n)$, which contains the trivial module, always has a nonzero degree of atypicality when $m, n \geq 1$.
 
Atypicality arises due to the presence of isotropic roots (i.e. roots of length zero) in the root system, which expand the notion of linkage beyond the orbit of the Weyl group. For $\osp(2m+1|2n)$, the degree of atypicality is an integer in the range $0$ to $\mathrm{min}(m, n)$, inclusive. In the integral case, any typical (i.e. degree of atypicality $0$) block can be reduced to the semisimple Lie algebra case via an equivalence of categories (cf. \cite{gorelik10.2307/827094}). Therefore, the new cases arise primarily when the degree of atypicality is nonzero.
\subsection{}
The problem of determining the irreducible representations that appear in a Jordan-H{\"o}lder series of a Verma module of a semisimple Lie algebra has a detailed history. For an integral weight, Kazhdan and Lusztig conjectured that these multiplicities could be determined in terms of certain recursively defined polynomials generated from the Weyl group of the semisimple Lie algebra (cf. \cite{Lusztig1979}). The Kazhdan-Lusztig conjecture was proven via geometric methods in the 1980s by Beilinson \& Bernstein (\cite{beilinsonbernstein81}) and Brylinski \& Kashiwara (\cite{brylinski1981kazhdan}).
 
Generalizing to the basic Lie superalgebra case has been difficult because the Weyl group no longer solely dictates linkage, but some progress has been made (cf. \cite{brundan2003kazhdan, Brundan_2016, cheng2011super, Cheng_2015}). An entirely different approach (and therefore solving the problem for certain semisimple Lie algebras in a novel way) was done for $\osp(l | 2n)$ by way of quantum symmetric pairs by Bao and Wang (cf. \cite{bao2017kazhdan, bao2018new}).
 
Nonetheless, these methods do not readily offer concrete multiplicities. By way of translation functors, we explicitly compute standard filtration formulae for projectives. This method is used to solve a similar problem for $\gl(3|1)$ and $\gl(2|2)$ in \cite{KANNAN2019231}, for $G(3)$ in \cite{chengwangg3}, and $D(2|1 ; \zeta)$ in \cite{cheng2019character}.
\subsection{} In this work, we use the tool of translation functors to determine the characters of projective modules in the BGG category $\OO$ for the orthosymplectic Lie superalgebras $\osp(3|4)$. Specifically, we explicitly determine the Verma multiplicities of standard filtration of projective modules in atypical blocks in $\OO$. There are infinitely many inequivalent atypical blocks. Then, BGG reciprocity allows us to convert these formulae to formulae for composition multiplicities, which we also explicitly state.
 
\subsection{} Our general approach to using translation functors is as follows. Given some projective cover $P_\lambda$ for which we wish to deduce Verma multiplicities, we find some $P_\mu$ with known Verma multiplicities and some finite-dimensional representation $N$ such that the Verma module $M_\lambda$ appears in a standard filtration of $P_\mu \otimes N$. If $\lambda$ is the lowest weight appearing among all the weights linked to $\lambda$ appearing in the Verma flag, then $P_\lambda$ is a direct summand for the projection of $P_\mu \otimes N$ on to the block corresponding to $\lambda$. In most cases, it is the only direct summand.
 
A particularly useful set of criteria for determining whether a summand is direct and for verifying indecomposability is stated in Proposition~\ref{filprop}. These criteria follow from similar criteria on tilting modules (cf. \cite{chengwangg3}) derived from the Super Jantzen sum formula (cf. \cite{musson2012lie}). Verifying indecomposability is a non-trivial step, as it is not evident whether or not translation functors yield an indecomposable projective. See \cref{conditions} and \cref{strategy} for explicit details and justification.
 
Our approach shows that in the cases we consider, standard filtrations always have Verma modules with multiplicity $1$ or $2$. By BGG reciprocity, these formulae determine the composition factors for Verma modules in $\OO$.
\subsection{}
In \cref{prelims}, we recall basic structure theorems for $\osp(3|4)$, fix a Cartan subalgebra, a root system, a fundamental system, and define linkage. Also, we recall the BGG category $\OO$, review relevant results in the super case, and offer conditions when Verma modules appear in the standard filtration of projective modules.
 
Section~\ref{sec5} contains our original results. We find standard filtration multiplicities for projective modules of atypical integral highest weight when $\g = \osp(3|4)$. These results are justified using the general facts in \cref{prelims} and the strategy of translation functors. 

Finally, in Section~\ref{sec7}, we compute the composition multiplicities of Verma modules of atypical integral highest weight for $\osp(3|4)$ using the BGG reciprocity.

\subsection*{\textbf{Acknowledgements}} This paper is the result of MIT PRIMES, a program that provides high-school students an opportunity to engage in research-level mathematics and in which the first author mentored the second. The first author would like to thank the Massachusetts Institute of Technology for financial support in the form of a MathWorks fellowship; also, this paper is based upon work supported by the National Science Foundation Graduate Research Fellowship Program under Grant No. 1842490 awarded to the first author. The authors would also like to thank the MIT PRIMES organizers for providing this opportunity. Finally, the authors thank David Vogan and Kevin Coulembier for answering some questions about the project, and Volodymyr Mazorchuk for suggesting revisions during the editorial process.

\section{Preliminaries}\label{prelims}
We shall recall elementary properties about the structure of the Lie superalgebra $\osp(3 | 4)$ and introduce some basic notations. For a more general reference on Lie superalgebras, see \cite{cheng2012dualities} or \cite{musson2012lie}.

\subsection{Basic definitions}
Suppose $V = \mathbb{C}^{3|4} = \mathbb{C}^{3} \oplus \mathbb{C}^4$. Let $\{\ov{1}, \ov{2}, \ov{3}\}$ and $\{1, 2, 3, 4\}$ parametrize the standard bases for the even and odd subspaces of $V$, $\mathbb{C}^{3}$ and $\mathbb{C}^{4}$, respectively. Denote
\[I(3,4) = \{\ov{1}, \ov{2}, \ov{3}; 1, 2, 3, 4\}\]
where we impose the total order

\[\ov{1} < \ov{2} < \ov{3} < 0 < 1 < 2 < 3 < 4.\]
Here, $0$ is introduced for convenience. The Lie superalgebra $\gl(3 | 4)$ is the Lie superalgebra of $7 \times 7$ matrices over $\mathbb{C}$ with bracket to be defined. The basis indexed by $I(3,4)$ for $V$ induces a basis for $\gl(3 | 4)$ given by $\{E_{ij} : i, j \in I(3,4) \}$, where $E_{ij}$ is the elementary matrix with a $0$ in every entry except for a $1$ in the $i$-th row and $j$-th column ($i,j \in I(3,4))$. The even subalgebra $\gl(3|4)_{\ov{0}}$ of $\gl(3 | 4)$ has a basis $\{E_{ij} : i, j < 0, \ i, j > 0 , \ i, j \in I(3,4)\}$ and the odd subspace $\gl(3|4)_{\ov{1}}$ has a basis $\{E_{ij} : i < 0 < j, \ j < 0 < i ,  \ i, j \in I(3,4) \}$. An element that is either purely even or purely odd is said to be homogeneous, and its parity (denoted $|\cdot|$) is $0$ or $1$, respectively. The Lie superbracket is defined on homogeneous elements $x,y \in \gl(3|4)$

\[
    [x,y] = xy - (-1)^{|x||y|}yx
\]
and extended by bilinearity. Define the supertranspose $x^{st}$ of an element $x \in \gl(3|4)$ in $(3|4)$-block form $x = \begin{pmatrix}a & b \\ c & d\end{pmatrix}$ by $x^{st} = \begin{pmatrix}a^t & c^t \\ -b^t & d^t\end{pmatrix}$, where $t$ denotes the regular transpose. Then, we define the Lie superalgebra $\osp(3 | 4)$ by stabilizing a non-degenerate even supersymmetric bilinear form as follows:

\[
    \osp(3 | 4) = \{g \in \gl(3|4) \ | g^{st}\mathfrak{J} + \mathfrak{J}g = 0\},
\]
where if $I_k$ is the $k\times k$ identity matrix, $\mathfrak{J}$ is the $7 \times 7$ matrix in the $(1|1|1|2|2)$-block form

\[
    \mathfrak{J} = 
    \begin{pmatrix}
    1 & 0 & 0 & 0 & 0 \\
    0 & 0 & 1 & 0 & 0 \\
    0 & 1 & 0 & 0 & 0 \\
    0 & 0 & 0 & 0 & I_2 \\
    0 & 0 & 0 & -I_2 & 0
    \end{pmatrix}.
\]
Then, we can realize $\osp(3 | 4)$ as the set of supermatrices of the following form:
\[
    \left(\begin{array}{ccc|cccc}
    0 & a_1 & a_2 & x_1 & x_2 & x_3 & x_4 \\
    -a_2 & h' & 0 & y_1 & y_2 & z_1 & z_2 \\
    -a_1 & 0 & -h' & z_3 & z_4 & y_3 & y_4 \\
    \hline 
    -x_3 & -y_3 & -z_1 & h_1 & b_1 & c_1 & c_2 \\
    -x_4 & -y_4 & -z_2 & b_2 & h_2 & c_2 & c_3 \\
    x_1 & z_3 & y_1 & d_1 & d_2 & -h_1 & -b_2 \\
    x_2 & z_4 & y_2 & d_2 & d_3 & -b_1 & -h_2 \\
    \end{array}\right).
\]
The even subalgebra $\osp(3|4)_{\ov{0}}$ (resp. odd subspace $\osp(3|4)_{\ov{1}}$) consists of the elements in $\osp(3|4)$ that are also in $\gl(3|4)_{\ov{0}}$ (resp. $\gl(3|4)_{\ov{1}}$). As a semisimple Lie algebra, $\osp(3|4)_{\ov{0}}$ is isomorphic to $\mathfrak{sp}(4)\oplus \mathfrak{so}(3)$.

Let $h_j = E_{j, j} - E_{2+j, 2+j}$ for $1 \leq j \leq 2$ and let $h' = E_{\ov{2}, \ov{2}} - E_{\ov{3}, \ov{3}}$, and let $\h$ denote the Cartan subalgebra of diagonal matrices in $\osp(3|4)$ given by:

\[
    \h = \mathbb{C}h_1\oplus\C h_2\oplus\C h'.
\]
Then, consider the dual basis for $\h^*$ given by $\{\delta_1, \delta_2, \epsilon\}$, where 

\[
   \delta_1(h_1) = 1, \delta_1(h_2) = \delta_1(h') = 0
\]
\[
   \delta_2(h_1) = 0, \delta_2(h_2) = 1, \delta_1(h') = 0
\]
\[
    \epsilon(h_1) = \epsilon(h_2) = 0, \epsilon(h') = 1.
\]
For convenience, we shall use the notation $\lambda = (a, b \ | \ c)$ to refer to the weight $\lambda = a\delta_1 + b\delta_2 + c\epsilon \in \h^*$ with $a, b, c \in \mathbb{C}$. Define a bilinear form $(\cdot, \cdot) : \h^* \times \h^* \rightarrow \mathbb{C}$ given by

\[
\begin{aligned}
 (\delta_j, \delta_k) = \delta_{jk}, \  (\epsilon, \epsilon) = -1 \\ 
 (\delta_j, \epsilon) = (\epsilon, \delta_j) = 0,
\end{aligned}
\]
where $1 \leq j, k \leq 2$ and then we extend by bilinearity. We can define the corresponding integral weight lattice $X$ in $\h^*$:
\[
X \coloneqq \Z\delta_1 \oplus \Z\delta_2 \oplus \Z\epsilon .
\]
Furthermore, with this choice of $\h$ we have a triangular decomposition $\osp(3 | 4) = \mathfrak{n}^- \oplus \h \oplus \mathfrak{n}^+$ of lower, diagonal, and upper triangular matrices, respectively, and root system $\Phi = \Phi_{\ov{0}} \cup \Phi_{\ov{1}}$ where the positive roots are given by
\[
\begin{aligned}
\Phi^+=\Phi_{\overline{0}}^+ \cup \Phi_{\overline{1}}^+=\{ 2\delta_1, 2\delta_2, \delta_1 \pm\delta_2, \epsilon \} \cup \{ \delta_1, \delta_2, \delta_1\pm\epsilon, \delta_2\pm\epsilon \},
\end{aligned}
\]
where $\Phi^+_{\oo} = \Phi^+ \cap \Phi_{\oo}$ denotes the positive even roots and $\Phi^+_{\ov{1}} = \Phi^+ \cap \Phi_{\ov{1}}$ denotes the positive odd roots. Call a root $\alpha \in \Phi$ isotropic if $(\alpha, \alpha) = 0$; positive isotropic roots are of the form $\delta_i \pm \epsilon$. Let

\[
\begin{aligned}
\Pi = \{\delta_1 - \delta_2, \epsilon \}\cup\{\delta_{2} - \epsilon\}, \\
\end{aligned}
\]
be a fundamental system for $\Phi$. We can establish the {Bruhat order} on $\h^*$ as follows. Let $\lambda, \mu \in \h^*$. We say $\lambda \geq \mu$ if $\lambda \sim \mu$ and $\lambda - \mu \in \mathbb{Z}_{\geq 0}\Pi$ (i.e the nonnegative sum of simple roots). 
\par
We can define for any $\alpha \in \Phi_{\oo}$ the corresponding coroot $\alpha^\vee \in \h$ such that for any $\lambda \in \h^*$
\[
\langle \lambda, \alpha^\vee \rangle = \frac{2(\lambda, \alpha)}{(\alpha, \alpha)}.
\]
The associated reflection $s_\alpha$ acts on $\h^*$ as expected: $s_\alpha(\lambda) = \lambda - \langle \lambda, \alpha^\vee \rangle \alpha$. The group generated by all even reflections is the Weyl group $ \mathcal{W} = \mathcal{W}_{\mathfrak{sp}(4)} \times \mathcal{W}_{\mathfrak{so}(3)} \cong (\mathbb{Z}_2^2 \rtimes S_2) \times \mathbb{Z}_2$. A simple description of its action on $\h^{*}$ is that $\mathcal{W}$ acts by signed permutations of the $\delta_j$'s or a sign swap of $\epsilon$ in the natural way. Also, notice that the Weyl group is not generated by the simple reflections.
 
The Weyl vector $\rho$ is defined as half the difference of the sum of the positive even roots and the sum of the positive odd roots. Explicitly, it is:
\begin{equation}\label{normWeyl}
\rho = \frac{1}{2}\delta_1 - \frac{1}{2}\delta_2 + \frac{1}{2}\epsilon.
\end{equation}
The dot-action $(\cdot)$ of $\mathcal{W}$ on $\h^*$ is given by $\rho$-shifting the regular action. Call a weight $\lambda$ dot-regular if $|\mathcal{W}\cdot\lambda| = |\mathcal{W}|$ and dot-singular otherwise. Notice that any vector in $X + \rho$ is half-integer and therefore not orthogonal to the non-isotropic odd roots $\delta_1$ and $\delta_2$.

Lastly, a weight $\lambda \in \h^*$ is said to be {antidominant} if $\langle \lambda + \rho, \alpha^\vee\rangle \not\in \mathbb{Z}_{> 0}$ and {dominant} if $\langle \lambda + \rho, \alpha^\vee\rangle \not\in \mathbb{Z}_{< 0}$ for all $\alpha \in \Phi^+_{\oo}$.
\subsection{Atypicality and linkage}\label{link}
The notion of linkage in the super case is similar to that of semisimple Lie algebras. However, the key distinction is that while blocks of modules in the semisimple Lie algebra case contain finitely many simple modules, odd roots allow for blocks in the super case to have infinitely many simple modules (see \cref{blocks} below). This arises because of a notion called atypicality.
 
Let $\h$ be the Cartan subalgebra of $\osp(3|4)$ and $\Phi$ be the root system as above. The {degree of atypicality} of $\lambda \in \h^*$, denoted $\#\lambda$, is the {maximum number of mutually orthogonal positive isotropic roots} $\alpha \in \Phi^+_{\one}$ such that $(\lambda + \rho, \alpha) = 0$. An element $\lambda \in \h^*$ is said to be {typical} (relative to $\Phi^+$) if $\#\lambda = 0$ and is atypical otherwise.
\par
We will mainly be interested in modules of integral highest weight $\lambda \in X$, frequently written as $\lambda = (a,b \ | \ c)-\rho$ where $a,b,c \in \mathbb{Z}+\frac{1}{2}$. In this case, atypicality can easily be read off and $\lambda$ is atypical of degree $1$ when $|c| = |a|$ or $|c| = |b|$ or both and typical otherwise. 
\par
A relation $\sim$ on $\h^*$ can be defined as following. We say $\lambda \sim \mu \ \lambda,\mu \in \h^*$ if there exist mutually orthogonal isotropic roots $\alpha_1,\alpha_2,\dots,\alpha_l$, complex numbers $c_1,c_2,\dots,c_l$, and an element $w \in \W$ satisfying:
\[
\mu + \rho = w\left(\lambda + \rho - \sum_{i=1}^{l}c_i\alpha_i\right), \ \ \ (\lambda + \rho, \alpha_i) = 0, i = 1\dots , l.
\]
The weights $\lambda$ and $\mu$ are said to be {linked} if $\lambda \sim \mu$. It can be shown that linkage is an equivalence relation, and weights of different degrees of atypicality are never linked. We note that we can restrict $\sim$ to $X$ (still denoted as $\sim$), where the $c_i$ can be taken to be integers. This is easily seen by noticing the following: $\mathcal{W}$ stabilizes $X$, $w\rho - \rho \in X$ for any $w \in \mathcal{W}$, the degree of atypicality is at most one, and that any non-integral multiple of an isotropic root is not in $X$.
\par
Here is an example. If $\lambda = (3/2, 1/2 | 1/2) - \rho$ and $\mu = (-3/2, 11/2 | 11/2) - \rho$, then both are atypical of degree $1$ and $\lambda \sim \mu$. On the other hand, if $\lambda = (1/2, 1/2 | 1/2 ) - \rho$ and $\mu = (7/2,7/2|7/2) - \rho$, they are still atypical of degree $1$ but are not linked. 
\subsection{Presentation of the Weyl Group}
Here, we take a closer look at the Weyl group $\mathcal{W}$, where the root system $\Phi$ and Cartan subalgebra $\h$ are as above. Denote by $r$ the reflection associated with $\delta_1 - \delta_2$, by $s$ the reflection associated with $2\delta_2$, and by $t$ the reflection associated with $\epsilon$. Then, the respective actions on $\h^*$ are given by permuting $\delta_1$ and $\delta_2$, negating $\delta_2$, and negating $\epsilon$. As a Coxeter group, the Weyl group has a presentation $\mathcal{W} = \langle r, s, t \ | \ r^2, s^2, t^2, (rs)^4, (rt)^2, (st)^2\rangle$. The first two reflections $r$ and $s$ generate $\mathcal{W}_{\mathfrak{sp}(4)}$, and $t$ generates $\mathcal{W}_{\mathfrak{so}(3)}$. We impose the Bruhat order on $\mathcal{W}$, writing $w' \leq w$ if a reduced word for $w'$ appears in some reduced word for $w$ for $w', w \in \mathcal{W}$. By the BGG theorem, this order is compatible with the partial order above on $\h^*$ in the sense that if $\lambda - \rho$ is typical, dot-regular, and antidominant, then $w' \leq w$ if and only if $w'\lambda \leq w\lambda$ (cf. \cite{humphreys2008representations}). $\mathcal{W}_{\mathfrak{sp}(4)}$ is dihedral and therefore the restricted Bruhat order is determined by comparing the lengths of elements. The Bruhat graph of $\mathcal{W}_{\mathfrak{sp}(4)}$ is given below: 

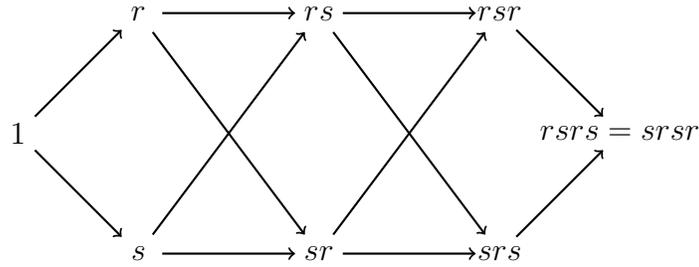
\begin{center}
\begin{tikzpicture}[thick,scale=0.8]
\begin{scope}[->,shorten >=9pt, shorten <= 9pt]

    \draw[->]{    (0,0) node {$1$}  -- (2,2) node {$r$}     };
    \draw[->]{    (0,0) node {}  -- (2,-2) node {$s$}     };
    \draw[->]{    (2,2) node {}  -- (5,2) node {$rs$}     };
    \draw[->]{    (2,-2) node {}  -- (5,-2) node {$sr$}     };
    \draw[->]{    (2,2) node {}  -- (5,-2) node {}     };
    \draw[->]{    (2,-2) node {}  -- (5,2) node {}     };    
    \draw[->]{    (5,2) node {}  -- (8,2) node {$rsr$}     };
    \draw[->]{    (5,-2) node {}  -- (8,-2) node {$srs$}     };
    \draw[->]{    (5,2) node {}  -- (8,-2) node {}     };
    \draw[->]{    (5,-2) node {}  -- (8,2) node {}     };    
    \draw[->]{    (8,2) node {}  -- (10,0) node {$rsrs = srsr$}     };
    \draw[->]{    (8,-2) node {}  -- (10,0) node {}     };
\end{scope}
\end{tikzpicture}    
\end{center}
Combined with the fact that $t$ is central, this makes clear the Bruhat order on $\mathcal{W}$. 

\subsection{The BGG category $\OO$}
From now on, let  $\g = \osp(3|4) = \g_{\oo} \oplus \g_{\one}$ with the standard associated bilinear form, root system, and triangular decomposition: $\g = \n^-\oplus\h\oplus\n^+$ and $\bo = \h\oplus\n^+$. Recall that the  (integral) {BGG category} $\OO$ is the full abelian subcategory of $U(\g)$-modules $M$ subject to the following three conditions:
\begin{enumerate}
\item $M$ is finitely generated.
\item $M$ is $\h$-semisimple: $M = \bigoplus_{\lambda \in X} M^\lambda$, where $M^\lambda = \{v \in M \ | \ h \cdot v = \lambda(h)v \ \ \mathrm{for \ all} \ h \in \h \}$.
\item $M$ is locally $\n^+$-finite: $U(\n^+)\cdot v$ is finite dimensional for all $v \in M$.
\end{enumerate}
Observe that the abelian quotient algebra $\bo /\n^+ \cong \h$. Thus, any $\lambda \in \h^*$ naturally defines a one-dimensional $\bo$-module with trivial $\n^+$-action, which we denote as $\mathbb{C}_\lambda$. Specifically, if $v \in \mathbb{C}_\lambda$, then $h\cdot v = \lambda(h) v$ for all $h \in \h$. Now, define
\[
M_\lambda \coloneqq U(\g)\otimes_{U(\bo)}\mathbb{C}_{\lambda - \rho},
\]
where $\rho$ is the Weyl vector. This is naturally a left $U(\g)$-module. This is called a {Verma module} of highest weight $\lambda - \rho$.
 
We let $L_\lambda$ denote the unique simple quotient of $M_\lambda$ of highest weight $\lambda - \rho$, and use the notation $[M_\mu : L_\lambda]$ to denote the multiplicity of $L_\lambda$ in a composition series of $M_\mu$. Such a series exists for all $M$ in $\OO$ as $\OO$ is both noetherian and artinian.
 
In the notation introduced in \cref{link}, if $\lambda = (a, b \ | \ c)$, write $M_{a, b| c}$ to denote $M_{\lambda}$ and $L_{a, b | c}$ to denote $L_{\lambda}$.

\subsection{Blocks in $\OO$}\label{blocks}
For $\lambda \in X + \rho$, let $[\lambda]$ denote the linkage class of $\lambda - \rho$ in $X/\sim$ and let $\OO_{[\lambda]}$ denote the Serre subcategory of $\OO$ generated by simple objects $L_\mu$ where $\mu \in X + \rho$ and $\mu - \rho  \sim \lambda - \rho$. Call $\OO_{[\lambda]}$ (a)typical if the linkage class $[\lambda]$ is (a)typical. 

\begin{prop}
The BGG category $\OO$ admits a decomposition

\[
 \OO = \bigoplus_{[\lambda] \in X/\sim} \OO_{[\lambda]}.
\]
This is precisely the decomposition of $\OO$ into indecomposable blocks.
\end{prop}

\begin{proof}
It is clear for $\lambda, \mu \in X + \rho$ that $\Hom_\OO(L_\lambda, L_\mu) = 0$ if $\lambda \neq \mu$. Furthermore, $\lambda - \rho$ and $\mu - \rho$ are linked if and only if they have the same central character (cf. Theorem 2.30 in \cite{cheng2012dualities}). Therefore,  $\lambda - \rho \sim \mu - \rho$ is a necessary condition for $\Ext^1_\OO(L_\lambda, L_\mu) \neq 0$. Since a composition series exists for all $M \in \OO$, we deduce $\OO$ admits the decomposition above. 
\par
It remains to show that each $\OO_{[\lambda]}$ is indecomposable. First, suppose the equivalence class $[\lambda]$ is typical. Then, the proof is similar to the semisimple Lie algebra case (cf. Section 1.13 in \cite{humphreys2008representations}). There are nonzero homomorphisms from $M_{s_\alpha\lambda}$ to $M_\lambda$ if $s_\alpha\lambda \leq \lambda$ and $\alpha \in \Phi_{\ov{0}}^+$ due to the existence of certain maximal vectors, where the image of the map lies in the unique maximal proper submodule $N_{\lambda}$ of $M_{\lambda}$. Quotienting out by $N_{s_\alpha\lambda}$ and its image $Q$ in $N_\lambda$ yields an embedding $L_{s_\alpha\lambda}$ into $M_\lambda/Q$, which is a highest weight module with quotient $L_\lambda$. Indecomposability of highest weight modules shows that $L_{s_\alpha\lambda}$ and $L_\lambda$ lie in the same block. Since the even reflections generate the Weyl group, the claim follows by induction (potentially swapping the roles of $s_\alpha\lambda$ and $\lambda$ if need be).
\par
Let us turn to the atypical case. First suppose $\lambda = (a, b| -b) \in X + \rho$ for fixed $a \in \Z_{\geq 0} + 1/2$ and $b \in \Z + 1/2$ and let $\gamma = \delta_2 - \epsilon$ be the simple isotropic root. Then, $(\lambda, \gamma) = 0$, so $\OO_{[\lambda]}$ is atypical. Then, there exists a nonzero homomorphism from $M_{\lambda - \gamma}$ to $M_\lambda$ (cf. Lemma 2.24 in \cite{cheng2012dualities}). By an argument like the one above, $L_{(a, c |-c)}$ and $L_{(a, b | -b)}$ lie in the same block for all $c \in \Z + 1/2$, and the argument directly shows $L_{w(a,c|-c)}$ lies in the same block as $L_{(a,c|-c)}$ for all $w \in \mathcal{W}$. This exhausts all weights in the linkage class associated to $\lambda - \rho$ and therefore shows the claim.
\end{proof}

Therefore, blocks in $\OO$ are indexed by linkage classes. In particular, each $a \in \mathbb{Z}_{\geq 0} + 1/2$ specifies a different atypical block $\mathcal{B}_{a} = \mathcal{O}_{[(a, b|b)]}$, with $b \in \mathbb{Z} + 1/2$. All integral atypical blocks are given this way. In particular, the principal block $\mathcal{B}_{1/2}$ contains the trivial module. 
\par
The typical blocks in $\OO$ are described by Gorelik (see section 8.5.1 in \cite{gorelik2002annihilation} and theorem 1.3.1 in \cite{gorelik10.2307/827094}). Because any $\rho$-shifted integral weight is strongly typical in the sense of \cite{gorelik10.2307/827094}, we get
\begin{prop}[Gorelik]\label{equiv}
Any typical block in $\OO$ is equivalent to a block in the BGG category $\OO$ of $\mathfrak{g}_{\ov{0}}$-modules of integral weights.
\end{prop}

\subsection{Key results in $\OO$}
The primary means by which the goals of this paper are achieved are by using translation functors. We restate the necessary results to justify our steps. This collection of results is justified in \cite[Chap. 1-3]{humphreys2008representations} for the BGG category $\OO$ for semisimple Lie algebras; similar arguments extend them to the BGG category $\OO$ of $\osp(3|4)$-modules.

We say a module $N \in \OO$ has a {standard filtration} or a {Verma flag} if there is a sequence of submodules $0 = N_0 \subset N_1 \subset N_2 \subset \cdots N_k = N$ such that each $N_i/N_{i-1} \ 1\leq i \leq k$ is isomorphic to a Verma module.  The number of times the Verma module $M_\lambda$ appears in a standard filtration of $N$ is denoted by $(N : M_\lambda)$.
 
It can be shown that the length and the Verma multiplicities in a standard filtration are independent of choice of a standard filtration. Therefore, the following informal notation to indicate a standard filtration of a module is useful. If $c_i = (N : M_{\lambda_i})$, we write
\[
N = c_1M_{\lambda_1} + c_2M_{\lambda_2} + \dots + c_kM_{\lambda_k}.
\]
 
Similarly, if $d_i = [N : L_{\mu_i}]$, we write
\[
N = d_1L_{\mu_1} + d_2L_{\mu_2} + \dots + d_kL_{\mu_k}.
\]

\begin{thm}\label{tran}
Let $N$ be a finite dimensional $U(\g)$-module. For any $\lambda \in \h^*$, the tensor module $T \coloneqq M_\lambda \otimes N$ has a standard filtration containing Verma modules of the form $M_{\lambda+\mu}$, where $\mu$ ranges over the weights of $N$, each occurring $\mathrm{dim} \ N^\mu$ times in the filtration.
\end{thm}

We let $P_\lambda$ denote the (unique) projective cover for $L_\lambda$ for all $\lambda \in \h^*$, that is the indecomposable projective such that $P_\lambda \twoheadrightarrow L_\lambda \rightarrow 0$. We recall the following facts about projectives.
\begin{enumerate}

\item All projectives have a standard filtration. \label{p0}
\item The category $\OO$ has enough projectives. \label{p1}
\item If $P = Q \oplus R$ with $P, Q, R \in \OO$, $P$ is projective if and only if $Q$ and $R$ are projective.
\item If $P \in \OO$ is projective and indecomposable, then $P \cong P_\lambda$ for some $\lambda \in \h^*$.
\item The Verma modules $M_\mu$ which appear in a standard filtration of $P_\lambda$ satisfy $\mu \geq \lambda$ in the Bruhat ordering, and $M_\lambda$ appears with multiplicity $1$. \label{p4}
\end{enumerate}
These facts yield the following lemma.
\begin{lem}\label{lowest}
If $\lambda - \rho$ is the lowest weight in a standard filtration of a projective object $P$, then $P_\lambda$ is a direct summand of $P$.
\end{lem}

The following proposition, which follows from Theorem \ref{tran}, is a critical part of our translation functor arguments.
\begin{prop}\label{sum}
If a projective $P$ has a standard filtration given by $P_\lambda = \sum_{\nu} M_{\nu}$, the $\nu$ not necessarily distinct, then for any finite-dimensional representation $N$ with weights $\mu$, the standard filtration for $P\otimes N$ is given by $\sum_{\nu}\sum_{\mu} M_{\nu + \mu}$, where $\mu$ appears in the sum with multiplicity given by $\mathrm{dim}\ N^\mu$.
\end{prop}
Knowing the Verma flag structure of typical projectives will be key in determining those of atypical projectives. We have the following lemmas.
\begin{lem}\label{typAndreg}
If $\lambda \in X + \rho$ is such that $\lambda - \rho$ is typical and dot-regular, then the Verma modules that appear in a standard filtration of $P_\lambda$ are of the form $M_{w\lambda}$, where $w \in \W$ such that $w\lambda \geq \lambda$, and each Verma module appears with multiplicity $1$.
\end{lem}
\begin{proof}
By Proposition \ref{equiv}, we have an equivalence of categories to the Lie algebra $\g_{\ov{0}} =  \mathfrak{sp}(4) \oplus \mathfrak{so}(3)$. Since the Weyl group $\mathcal{W}$ is the product of dihedral groups, it is well known in this case that the Kazhdan-Lusztig polynomials are all monomials (and for our particular case it can be directly verified by computation). The result then follows by the Kazhdan-Lusztig conjecture.
\end{proof}
The lemma also extends to typical and dot-singular weights.
\begin{lem}\label{typAndsing}
Let $\lambda \in X + \rho$ be such that $\lambda - \rho$ is a typical anti-dominant dot-singular weight. Let $\mathcal{W}^{\lambda}$ be a minimal set of left-coset representatives of $\mathcal{W}/\mathcal{W}_{\lambda}$, where $\mathcal{W}_{\lambda} = \{w \in \mathcal{W} \ | \  w\lambda = \lambda\}$. Then, if $\sigma \in \mathcal{W}^{\lambda}$,

\begin{equation}\label{singform}
P_{\sigma\lambda} = \sum_{\tau \geq \sigma, \tau \in \mathcal{W}^{\lambda}} M_{\tau\lambda}.
\end{equation}
\end{lem}
\begin{proof}
The proof is analogous to that of Lemma 3.5 in \cite{chengwangg3}. Since $\lambda = a\delta_1 + b\delta_2 + c\epsilon$ with $a, b, c \in \Z + \frac{1}{2}$ is singular and in particular $c \neq 0$, changing the sign of $c$ does not stabilize $\lambda$. Hence, the action of $\mathcal{W}_{\mathfrak{so}(3)}$ is always regular. Therefore, $\{e\} \neq \mathcal{W}_{\lambda} \subseteq \mathcal{W}_{\mathfrak{sp}(4)}$. The central character corresponding to the integral weight $\lambda - \rho$ is strongly typical in the sense of Gorelik (cf. \cite{gorelik2002annihilation}) and by Proposition \ref{equiv} we have an equivalence of categories between the block containing the irreducible module $L_{\lambda}$ and a singular integral block of $\mathfrak{sp}(4)\oplus \mathfrak{so}(3)$-modules.
 
Since the action of  $\mathcal{W}_{\mathfrak{so}(3)}$ is regular, it suffices to check the analog of \eqref{singform} for a singular integral block of $\mathcal{W}_{\mathfrak{sp}(4)}$ modules. Since the corresponding Weyl group is dihedral and the Kazhdan-Lusztig polynomials are monomials, the lemma follows by Theorem 3.11.4 in \cite{beilinson1996koszul}.
\end{proof}

Lastly, we recall BGG reciprocity.
\begin{equation}\label{BGGrecip}
(P_\lambda : M _\mu) = [M_\mu : L_\lambda], \ \ \lambda, \mu \in \h^* .
\end{equation}
\subsection{Some representations of $\osp(3|4)$}
The strategy of using translation functors involves choosing appropriate representations to tensor with projective modules to produce new modules.
 
The simplest module we use is the seven-dimensional {natural representation} $V = \mathbb{C}^{3|4}$ of $\osp(3|4)$. We also use the second symmetric power $S^2 V$ of the natural representation (call it the symmetric-squared of the natural) and the adjoint representation $\mathfrak{g}$. In general, the $k$-th symmetric power of a vector superspace $W = W_{\oo}\oplus W_{\one}$ is defined as:
\[
S^k(W) \coloneqq \bigoplus_{i+j=k} \left(S^i (W_{\oo}) \otimes \Lambda^i (W_\one)\right),
\]
where $S^i$ and $\Lambda^j$ acting on vector spaces are the $i$-th symmetric power and $j$-th exterior power in the traditional sense, respectively. The natural representation has dimension $7$, the symmetric-squared of the natural has dimension $24$, and the adjoint representation has dimension $25$.

\subsection{Conditions for nonzero Verma flag multiplicities in projective modules}\label{conditions}
We have the following proposition, which uses BGG reciprocity to reformulate the conditions for tilting modules in \cite[Proposition 2.2]{chengwangg3} as conditions for projective modules.
\begin{prop}\label{filprop}

Suppose that $\lambda \in X + \rho, \alpha_i \in \Phi_{\bar{0}}^+, 1 \leq i \leq k,$ and $\beta, \gamma \in \Phi_{\bar{1}}^+$. Let $w = s_{\alpha_k}s_{\alpha_{k-1}}\cdots s_{\alpha_1} \in \W$.
	\begin{enumerate}
	    \item Suppose that $\langle\lambda, \alpha_1^\vee\rangle < 0$. Then $(P_{\lambda} : M_{s_{\alpha_1}\lambda}) > 0$. \label{1}

	    \item Suppose that $\langle s_{\alpha_{i-1}}\cdots s_{\alpha_1}\lambda, \alpha_i^\vee\rangle < 0$ for all  $i \in {1,2,\dots,k}$. then $(P_{\lambda} : M_{w\lambda}) > 0$. \label{2}

	    \item  Suppose that $(\lambda, \beta) = 0$. Then $(P_{\lambda} : M_{\lambda + \beta}) > 0$. \label{3}

	    \item Suppose that $(\lambda, \beta) = 0$ and  $\langle s_{\alpha_{i-1}}\cdots s_{\alpha_1}(\lambda + \beta), \alpha_i^\vee\rangle < 0$ for all $i \in {1,2,\dots,k}$. Then $(P_{\lambda} : M_{w(\lambda + \beta)}) > 0$. \label{4}

	    \item Suppose that $(\lambda, \beta) = (\lambda + \beta, \gamma) = 0$ and $\mathrm{ht}(\beta) < \mathrm{ht}(\gamma)$. Then $(P_{\lambda} : M_{\lambda + \beta + \gamma}) > 0$. \label{5}

	    \item Suppose that $(\lambda, \beta) = (\lambda + \beta, \gamma) = 0$, $\mathrm{ht}(\beta) < \mathrm{ht}(\gamma)$, and $\langle s_{\alpha_{i-1}}\cdots s_{\alpha_1}(\lambda + \beta + \gamma), \alpha_i^\vee\rangle < 0$ for all $i \in {1,2,\dots,k}$. Then $(P_{\lambda} : M_{w(\lambda + \beta + \gamma)}) > 0$. \label{6}
\end{enumerate}
\end{prop}

\begin{cor}\label{corlen}
Suppose $\lambda \in X + \rho$ is atypical such that $\lambda - \rho$ is atypical. Then $P_\lambda$ must have a Verma flag of length greater than $1$.
\end{cor}

\subsection{Strategy}\label{strategy}
Given a $\lambda \in X + \rho$ such that $\lambda - \rho$ is atypical, we seek to deduce the standard filtration formula of $P_{\lambda}$. To do so, we choose a $\mu \in X + \rho$ such that we know a standard filtration for $P_{\mu}$. This is often accomplished by letting $\mu \coloneqq \lambda - \nu$, where $\nu$ is a weight (often the lowest) in some finite-dimensional representation $W$ such that $\mu - \rho$ is typical; Lemma~\ref{typAndreg} and Lemma~\ref{typAndsing} tell us the structure of $P_{\mu}$. Proposition~\ref{sum} can be used to deduce the Verma modules which appear in a standard filtration of the projective $P_{\mu} \otimes W$, which must include $M_{\lambda}$. Our next step is to project $P_{\mu} \otimes W$ onto the block corresponding to the linkage class of $\lambda - \rho$. We denote the resulting projection as $\mathrm{pr}_\lambda(P_\mu \otimes W)$. By Lemma~\ref{lowest}, if $M_{\lambda}$ has the lowest weight of all the Verma modules in the standard filtration of the projection, $P_\lambda$ must appear in that projection as a direct summand. The projection itself is done by collecting all Verma modules in the standard filtration whose weights are linked to $\lambda - \rho$.
 
In this projection, we apply Proposition \ref{filprop} to see which Verma modules appear in the standard filtration of $P_\lambda$. These necessarily appear in the projection because $P_\lambda$ is a direct summand. Then, we generally try to argue that there is no other direct summand (i.e. $P_\lambda$ is the projection). This is often done by showing that no other indecomposable projective can appear in the projection, since there are not enough terms. In certain special cases, this method fails, and we get two possible standard filtrations of $P_\lambda$. To determine which one is correct, we generally show that one of them is not a projective.
 
For convenience, we introduce the following notation which we use extensively in the presentation of our results and proofs to save space and improve clarity. Let $\lambda \in X+\rho$ be such that $\lambda - \rho$ is anti-dominant. Let $\mathcal{W}^{\lambda}$ be a minimal set of left-coset representatives of $\mathcal{W}/\mathcal{W}_{\lambda}$, where $\mathcal{W}_{\lambda} = \{w \in \mathcal{W} \ | \  w\lambda = \lambda\}$. Then, if $\sigma\in\mathcal{W}^{\lambda}$, we denote 
\[
    \sum_{\tau \geq \sigma, \tau \in \mathcal{W}^{\lambda}} M_{\tau\lambda}
\]
by
\[
    \sum M_{\sigma\lambda}.
\]
For example, we may write 
\[
    M_{\frac{1}{2},-\frac{3}{2}|\frac{1}{2}} + M_{\frac{1}{2},\frac{3}{2}|\frac{1}{2}} + M_{\frac{3}{2},-\frac{1}{2}|\frac{1}{2}} + M_{\frac{3}{2},\frac{1}{2}|\frac{1}{2}}
\]
as
\[
    \sum M_{\frac{1}{2},-\frac{3}{2}|\frac{1}{2}}.
\]

\section{Character Formulae for \texorpdfstring{$\osp(3|4)$}{osp(3|4)}}\label{sec5}
In this section, we determine Verma multiplicities for standard filtration formulae for projective covers of simple modules of $\osp(3|4)$ with integral, atypical highest weight.
\subsection{Results}
Let $\g = \osp(3|4)$ have the standard choices of Cartan subalgebra, bilinear form, root system, positive, and fundamental system as described in \cref{prelims}. Recall the notation described in \cref{link} to describe a weight in $\h^*$. We have the following Theorems \ref{thm:spo431} to \ref{thm:spo434} that describe standard filtrations of projectives in these blocks. We provide the proofs after presenting all four theorems.

\begin{thm}\label{thm:spo431}
    Let $\lambda-\rho = (a,b \ | \ c)-\rho$ be an atypical weight with $a,b,c \in \frac{1}{2} + \mathbb{Z}$, $a,b>0$, and $c \in \{\pm a, \pm b\}$. The projective covers $P_{\lambda}$ have the following Verma flag formulae.
    
    \begin{enumerate}[label=(\arabic*), ref=\arabic*]
        \item Suppose that $a>b>0$.
        \begin{enumerate}[label=(\theenumi.\arabic*), ref=\arabic*]
            \item When $c=a$, we have 
            \[P_{a,b|a} = M_{a,b|a} + M_{a+1,b|a+1}.\]
            
            \item When $c=-a$, we have
            \[P_{a,b|-a} = M_{a,b|-a} + M_{a,b|a} + M_{a+1,b|-a-1} + M_{a+1,b|a+1}.\]
            
            \item When $c=b$, we have
            \[P_{a,b|b} = M_{a,b|b} + M_{a,b+1|b+1}\] 
            for $b<a-1$, and
            \[P_{a,a-1|a-1} = M_{a,a-1|a-1} + M_{a,a|a} + M_{a+1,a|a+1}.\]
            
            \item When $c=-b$, we have
            \[P_{a,b|-b} = M_{a,b|-b} + M_{a,b|b} + M_{a,b+1|-b-1} + M_{a,b+1|b+1}\]
            for $b<a-1$, and
            \begin{align*}
                P_{a,a-1|-a+1} &= M_{a,a-1|-a+1} + M_{a,a-1|a-1} + M_{a,a|-a} + M_{a,a|a} \\
                &+ M_{a+1,a|-a-1} + M_{a+1,a|a+1}.
            \end{align*}
        \end{enumerate}

        \item Suppose that $b>a>0$.
        \begin{enumerate}[label=(\theenumi.\arabic*), ref=\arabic*]
            \item When $c=a$, we have
            \[P_{a,b|a} = M_{a,b|a} + M_{b,a|a} + M_{a+1,b|a+1} + M_{b,a+1|a+1}\]
            for $b>a+1$, and
            \[P_{a,a+1|a} = M_{a,a+1|a} + M_{a+1,a|a} + M_{a+1,a+1|a+1}.\]
            
            \item When $c=-a$, we have
            \begin{align*}
                P_{a,b|-a} &= M_{a,b|-a} + M_{a,b|a} + M_{b,a|-a} + M_{b,a|a} \\
                &+ M_{a+1,b|-a-1} + M_{a+1,b|a+1} + M_{b,a+1|-a-1} + M_{b,a+1|a+1}
            \end{align*}
            for $b>a+1$, and 
            \begin{align*}
                P_{a,a+1|-a} &= M_{a,a+1|-a} + M_{a,a+1|a} + M_{a+1,a|-a} 
                + M_{a+1,a|a} \\
                &+ M_{a+1,a+1|-a-1} + M_{a+1,a+1|a+1}.
            \end{align*}
            
            \item When $c=b$, we have
            \[P_{a,b|b} = M_{a,b|b} + M_{b,a|b} + M_{a,b+1|b+1} + M_{b+1,a|b+1}.\]
            
            \item When $c=-b$, we have
            \begin{align*}
                P_{a,b|-b} 
                &= M_{a,b|-b} + M_{a,b|b} + M_{b,a|-b} + M_{b,a|b} \\
                &+ M_{a,b+1|-b-1} + M_{a,b+1|b+1} + M_{b+1,a|-b-1} + M_{b+1,a|b+1} \\
                &= \sum M_{a,b|-b} + \sum M_{a,b+1|-b-1}.
            \end{align*}
        \end{enumerate}

        \item Suppose that $a=b>0$.
        \begin{enumerate}[label=(\theenumi.\arabic*), ref=\arabic*]
            \item When $c=a$, we have
            \[P_{a,a|a} = M_{a,a|a} + M_{a,a+1|a+1} + M_{a+1,a|a+1}.\]
            
            \item When $c=-a$, we have
            \begin{align*}
                P_{a,a|-a} &= M_{a,a|-a} + M_{a,a|a} + M_{a,a+1|-a-1} + M_{a,a+1|a+1} \\
                &+ M_{a+1,a|-a-1} + M_{a+1,a|a+1}.
            \end{align*}
        \end{enumerate}
    \end{enumerate}
\end{thm}

\begin{thm}\label{thm:spo432}
    Let $\lambda-\rho = (a,b \ | \ c)-\rho$ be an atypical weight with $a,b,c \in \frac{1}{2} + \mathbb{Z}$, $a>0>b$, and $c \in \{\pm a, \pm b\}$. The projective covers $P_{\lambda}$ have the following Verma flag formulae.
    
    \begin{enumerate}[label=(\arabic*), ref=\arabic*]        
        \item Suppose that $a>-b>0$
        \begin{enumerate}[label=(\theenumi.\arabic*), ref=\arabic*] 
            \item When $c=a$,
            \[P_{a,b|a} = M_{a,b|a} + M_{a,-b|a} + M_{a+1,b|a+1} + M_{a+1,-b|a+1}.\]
            
            \item When $c=-a$,
            \begin{align*}
                P_{a,b|-a} &= M_{a,b|-a} + M_{a,b|a} + M_{a,-b|-a} + M_{a,-b|a} \\
                &+ M_{a+1,b|-a-1} + M_{a+1,b|a+1} + M_{a+1,-b|-a-1} + M_{a+1,-b|a+1}.
            \end{align*}
            
            \item When $c=-b$,
            \[P_{a,b|-b} = M_{a,b|-b} + M_{a,-b|-b} + M_{a,b+1|-b-1} + M_{a,-b-1|-b-1}\]
            for $b<-\frac{1}{2}$, and
            \[P_{a,-\frac{1}{2}|\frac{1}{2}} = M_{a,-\frac{1}{2}|\frac{1}{2}} + M_{a,\frac{1}{2}|\frac{1}{2}} + M_{a,\frac{1}{2}|-\frac{1}{2}} + M_{a,\frac{3}{2}|\frac{3}{2}}\]
            for $a>\frac{3}{2}$, and
            \begin{align*}
                P_{\frac{3}{2},-\frac{1}{2}|\frac{1}{2}} 
                &= M_{\frac{3}{2},-\frac{1}{2}|\frac{1}{2}} 
                + M_{\frac{3}{2},\frac{1}{2}|\frac{1}{2}} 
                + M_{\frac{3}{2},\frac{1}{2}|-\frac{1}{2}} 
                + M_{\frac{3}{2},\frac{3}{2}|\frac{3}{2}} \\
                &+ M_{\frac{5}{2},\frac{3}{2}|\frac{5}{2}}.
            \end{align*}
            
            \item When $c=b$,
            \begin{align*}
                P_{a,b|b} 
                &= M_{a,b|b} + M_{a,b|-b} + M_{a,-b|b} + M_{a,-b|-b} \\
                &+ M_{a,b+1|b+1} + M_{a,b+1|-b-1} + M_{a,-b-1|b+1} + M_{a,-b-1|-b-1}
            \end{align*}
            for $b<-\frac{1}{2}$, and
            \begin{align*}
                P_{a,-\frac{1}{2}|-\frac{1}{2}} 
                &= M_{a,-\frac{1}{2}|-\frac{1}{2}} + M_{a,-\frac{1}{2}|\frac{1}{2}} + M_{a,\frac{1}{2}|-\frac{1}{2}} + M_{a,\frac{1}{2}|\frac{1}{2}} \\
                &= \sum M_{a,-\frac{1}{2}|-\frac{1}{2}}.
            \end{align*}
        \end{enumerate}
        
        \item Suppose that $-b>a>0$.
        \begin{enumerate}[label=(\theenumi.\arabic*), ref=\arabic*] 
            \item When $c=a$, we have
            \[P_{a,b|a} = \sum M_{a,b|a} + \sum M_{a+1,b|a+1}.\]
            
            \item When $c=-a$, we have
            \[P_{a,b|-a} = \sum M_{a,b|-a} + \sum M_{a+1,b|-a-1}.\]
            
            \item When $c=-b$, we have
            \[P_{a,b|-b} = \sum M_{a,b|-b} + \sum M_{a,b+1|-b-1}.\]
            
            \item When $c=b$, we have
            \[P_{a,b|b} = \sum M_{a,b|b} + \sum M_{a,b+1|b+1}.\]
        \end{enumerate}

        \item Suppose that $a=-b>0$.
        \begin{enumerate}[label=(\theenumi.\arabic*), ref=\arabic*] 
            \item When $c=a$, we have
            \begin{align*}
                P_{a,-a|a} 
                &= M_{a,-a|a} + M_{a,a|a} + M_{a,-a+1|a-1} + M_{a,a-1|a-1} \\
                &+ M_{a+1,-a|a+1} + M_{a+1,a|a+1}
            \end{align*}
            for $a>\frac{1}{2}$, and
            \begin{align*}
                P_{\frac{1}{2},-\frac{1}{2}|\frac{1}{2}} 
                &= M_{\frac{1}{2},-\frac{1}{2}|\frac{1}{2}} 
                + M_{\frac{1}{2},\frac{1}{2}|\frac{1}{2}} 
                + M_{\frac{1}{2},\frac{1}{2}|-\frac{1}{2}} 
                + M_{\frac{1}{2},\frac{3}{2}|\frac{3}{2}} \\
                &+ M_{\frac{3}{2},-\frac{1}{2}|\frac{3}{2}}
                + M_{\frac{3}{2},\frac{1}{2}|-\frac{3}{2}}
                + 2M_{\frac{3}{2},\frac{1}{2}|\frac{3}{2}}
                + M_{\frac{5}{2},\frac{1}{2}|\frac{5}{2}}.
            \end{align*}
            
            \item When $c=-a$, we have
            \[P_{a,-a|-a} = \sum M_{a,-a|-a} + \sum M_{a,-a+1|-a+1} +\sum M_{a+1,-a|-a-1}\]
            for $a>\frac{1}{2}$, and
            \[P_{\frac{1}{2},-\frac{1}{2}|-\frac{1}{2}} = \sum M_{\frac{1}{2},-\frac{1}{2}|-\frac{1}{2}} + \sum M_{\frac{3}{2},-\frac{1}{2}|-\frac{3}{2}}.\]
        \end{enumerate}
    \end{enumerate}
\end{thm}

\begin{thm}\label{thm:spo433}
    Let $\lambda-\rho = (a,b \ | \ c)-\rho$ be an atypical weight with $a,b,c \in \frac{1}{2} + \mathbb{Z}$, $b>0>a$, and $c \in \{\pm a, \pm b\}$. The projective covers $P_{\lambda}$ have the following Verma flag formulae.
    
    \begin{enumerate}[label=(\arabic*), ref=\arabic*]        
        \item Suppose that $a<-b<0$.
        \begin{enumerate}[label=(\theenumi.\arabic*), ref=\arabic*] 
            \item When $c=-a$, we have 
            \[P_{a,b|-a} = \sum M_{a,b|-a} + \sum M_{a+1,b|-a-1}.\]
            
            \item When $c=a$, we have
            \[P_{a,b|a} = \sum M_{a,b|a} + \sum M_{a+1,b|a+1}.\]
            
            \item When $c=b$, we have
            \[P_{a,b|b} = \sum M_{a,b|b} + \sum M_{a,b+1|b+1}.\]
            
            \item When $c=-b$, we have
            \[P_{a,b|-b} = \sum M_{a,b|-b} + \sum M_{a,b+1|-b-1}.\]
        \end{enumerate}

        \item Suppose that $-b<a<0$.
        \begin{enumerate}[label=(\theenumi.\arabic*), ref=\arabic*] 
            \item When $c=-a$, we have 
            \[P_{a,b|-a} = \sum M_{a,b|-a} + \sum M_{a+1,b|-a-1}\]
            for $a<-\frac{1}{2}$, and 
            \begin{align*}
                P_{-\frac{1}{2}, b|\frac{1}{2}}
                &= \sum M_{-\frac{1}{2}, b|\frac{1}{2}}\\
                &+ M_{\frac{1}{2}, b|-\frac{1}{2}} + M_{b, \frac{1}{2}|-\frac{1}{2}} + M_{\frac{3}{2}, b|\frac{3}{2}} + M_{b, \frac{3}{2}|\frac{3}{2}}
            \end{align*}
            for $b>\frac{3}{2}$, and 
            \begin{align*}
                P_{-\frac{1}{2}, \frac{3}{2}|\frac{1}{2}}
                &= \sum M_{-\frac{1}{2}, \frac{3}{2}|\frac{1}{2}}\\
                &+ M_{\frac{1}{2}, \frac{3}{2}|-\frac{1}{2}} + M_{\frac{3}{2}, \frac{1}{2}|-\frac{1}{2}} + M_{\frac{3}{2}, \frac{3}{2}|\frac{3}{2}}.
            \end{align*}
            
            \item When $c=a$, we have
            \[P_{a,b|a} = \sum M_{a,b|a} + \sum M_{a+1,b|a+1}\]
            for $a<-\frac{1}{2}$, and 
            \[P_{-\frac{1}{2}, b|-\frac{1}{2}} = \sum M_{-\frac{1}{2}, b|-\frac{1}{2}}.\]
            
            \item When $c=b$, we have
            \[P_{a,b|b} = \sum M_{a,b|b} + \sum M_{a,b+1|b+1}.\]
            
            \item When $c=-b$, we have
            \[P_{a,b|-b} = \sum M_{a,b|-b} + \sum M_{a,b+1|-b-1}.\]
        \end{enumerate}

        \item Suppose that $a=-b<0$.
        \begin{enumerate}[label=(\theenumi.\arabic*), ref=\arabic*] 
            \item When $c=-a$, we have
            \begin{align*}
                P_{a,-a|-a} &= M_{a,-a|-a} + M_{-a,a|-a} + 2M_{-a,-a|-a} \\
                &+ M_{a,-a+1|-a+1} + M_{-a,-a+1|-a+1} + M_{-a+1,a|-a+1} + M_{-a+1,-a|-a+1} \\
                &+ M_{a+1,-a|-a-1} + M_{-a-1,-a|-a-1} + M_{-a,a+1|-a-1} + M_{-a,-a-1|-a-1} \\
                &= \sum M_{a,-a|-a} + M_{-a,-a|-a} + \sum M_{a,-a+1|-a+1} + \sum M_{a+1,-a|-a-1}
            \end{align*}
            for $a<-\frac{1}{2}$, and 
            \begin{align*}
                P_{-\frac{1}{2},\frac{1}{2}|\frac{1}{2}} 
                &= M_{-\frac{1}{2},\frac{1}{2}|\frac{1}{2}} 
                + M_{\frac{1}{2},-\frac{1}{2}|\frac{1}{2}} 
                + M_{\frac{1}{2},\frac{1}{2}|\frac{1}{2}} 
                + M_{\frac{1}{2},\frac{1}{2}|-\frac{1}{2}} \\
                &+ M_{-\frac{1}{2},\frac{3}{2}|\frac{3}{2}} 
                + M_{\frac{1}{2},\frac{3}{2}|\frac{3}{2}} 
                + M_{\frac{3}{2},-\frac{1}{2}|\frac{1}{2}} 
                + M_{\frac{3}{2},\frac{1}{2}|\frac{1}{2}} \\
                &= \sum M_{-\frac{1}{2},\frac{1}{2}|\frac{1}{2}} + M_{\frac{1}{2},\frac{1}{2}|-\frac{1}{2}} + \sum M_{-\frac{1}{2},\frac{3}{2}|\frac{3}{2}}.
            \end{align*}
            
            \item When $c=a$, we have
            \begin{align*}
                P_{a,-a|a} 
                &= \sum M_{a,-a|a} + M_{-a,-a|a} + M_{-a,-a|-a} \\
                &+ \sum M_{a,-a+1|a-1} + \sum M_{a+1,-a|a+1}
            \end{align*}
            for $a<-\frac{1}{2}$, and 
            \[P_{-\frac{1}{2},\frac{1}{2}|-\frac{1}{2}} = \sum M_{-\frac{1}{2},\frac{1}{2}|-\frac{1}{2}} + M_{\frac{1}{2},\frac{1}{2}|-\frac{1}{2}} + M_{\frac{1}{2},\frac{1}{2}|\frac{1}{2}} + \sum M_{-\frac{1}{2},\frac{3}{2}|-\frac{3}{2}}.\]
        \end{enumerate}
    \end{enumerate}
\end{thm}

\begin{thm}\label{thm:spo434}
    Let $\lambda-\rho = (a,b \ | \ c)-\rho$ be an atypical weight with $a,b,c \in \frac{1}{2} + \mathbb{Z}$, $a,b<0$, and $c \in \{\pm a, \pm b\}$. The projective covers $P_{\lambda}$ have the following Verma flag formulae.
    
    \begin{enumerate}[label=(\arabic*), ref=\arabic*]        
        \item Suppose that $a<b<0$.
        \begin{enumerate}[label=(\theenumi.\arabic*), ref=\arabic*]     
            \item When $c=-a$, we have
            \[P_{a,b|-a} = \sum M_{a,b|-a} + \sum M_{a+1,b|-a-1}.\]
            
            \item When $c=a$, we have
            \[P_{a,b|a} = \sum M_{a,b|a} + \sum M_{a+1,b|a+1}.\]
            
            \item When $c=-b$, we have
            \[P_{a,b|-b} = \sum M_{a,b|-b} + \sum M_{a,b+1|-b-1}\]
            for $b<-\frac{1}{2}$, and
            \begin{align*}
                P_{a,-\frac{1}{2}|\frac{1}{2}}
                &= \sum M_{a,-\frac{1}{2}|\frac{1}{2}} + \sum M_{-\frac{1}{2}, -a|\frac{1}{2}} \\
                &+ M_{a,\frac{1}{2}|-\frac{1}{2}} + M_{-\frac{1}{2},-a|-\frac{1}{2}} + M_{\frac{1}{2},a|-\frac{1}{2}} + M_{\frac{1}{2},-a|-\frac{1}{2}} \\
                &+ M_{-a,-\frac{1}{2}|-\frac{1}{2}} + M_{-a,\frac{1}{2}|-\frac{1}{2}} + \sum M_{a,\frac{3}{2}|\frac{3}{2}}.
            \end{align*}
            
            \item When $c=b$, we have
            \[P_{a,b|b} = \sum M_{a,b|b} + \sum M_{a,b+1|b+1}\]
            for $b<-\frac{1}{2}$, and
            \[P_{a,-\frac{1}{2}|-\frac{1}{2}} = \sum M_{a,-\frac{1}{2}|-\frac{1}{2}}.\]
        \end{enumerate}

        \item Suppose that $b<a<0$.
        \begin{enumerate}[label=(\theenumi.\arabic*), ref=\arabic*]     
        \item When $c=-a$, we have
        \[P_{a,b|-a} = \sum M_{a,b|-a} + \sum M_{a+1,b|-a-1}\]
        for $a<-\frac{1}{2}$, and 
        \begin{align*}
            P_{-\frac{1}{2},b|\frac{1}{2}}
            &= \sum M_{-\frac{1}{2},b|\frac{1}{2}} + M_{-b,-\frac{1}{2}|\frac{1}{2}} + M_{-b,\frac{1}{2}|\frac{1}{2}} \\
            &+ M_{\frac{1}{2},b|-\frac{1}{2}} + M_{\frac{1}{2},-b|-\frac{1}{2}} + M_{-b, -\frac{1}{2}|-\frac{1}{2}} + M_{-b, \frac{1}{2}|-\frac{1}{2}} \\
            &+ \sum M_{\frac{3}{2},b|\frac{3}{2}}.
        \end{align*}
            
        \item When $c=a$, we have
        \[P_{a,b|a} = \sum M_{a,b|a} + \sum M_{a+1,b|a+1}\]
        for $a<-\frac{1}{2}$, and 
        \[P_{-\frac{1}{2},b|-\frac{1}{2}} = \sum M_{-\frac{1}{2},b|-\frac{1}{2}}.\]
        
        \item When $c=-b$, we have
        \[P_{a,b|-b} = \sum M_{a,b|-b} + \sum M_{a,b+1|-b-1}\]
        for $b<a-1$, and
        \begin{align*}
            P_{a,a-1|-a+1} 
            &= \sum M_{a,a-1|-a+1} + \sum M_{a+1,a|-a-1} \\
            &+ \sum M_{a,a|-a} + \sum M_{-a,a|-a}
        \end{align*}
        for $a<-\frac{1}{2}$, and
        \begin{align*}
            P_{-\frac{1}{2},-\frac{3}{2}|\frac{3}{2}} 
            &= \sum M_{-\frac{1}{2},-\frac{3}{2}|\frac{3}{2}} + M_{\frac{3}{2},-\frac{1}{2}|\frac{3}{2}} + M_{\frac{3}{2},\frac{1}{2}|\frac{3}{2}} \\
            &+ \sum M_{-\frac{1}{2},-\frac{1}{2}|\frac{1}{2}} + \sum M_{\frac{1}{2},-\frac{1}{2}|-\frac{1}{2}}.
        \end{align*}
        
        \item When $c=b$, we have
        \[P_{a,b|b} = \sum M_{a,b|b} + \sum M_{a,b+1|b+1}\]
        for $b<a-1$, and
        \begin{align*}
            P_{a,a-1|a-1} 
            &= \sum M_{a,a-1|a-1} + \sum M_{a+1,a|a+1} \\
            &+ \sum M_{a,a|a} + \sum M_{-a,a|a}
        \end{align*}
        for $a<-\frac{1}{2}$, and
        \[P_{-\frac{1}{2},-\frac{3}{2}|-\frac{3}{2}} = \sum M_{-\frac{1}{2},-\frac{3}{2}|-\frac{3}{2}} + \sum M_{-\frac{1}{2},-\frac{1}{2}|-\frac{1}{2}} + \sum M_{\frac{1}{2},-\frac{1}{2}|-\frac{1}{2}}.\]
        \end{enumerate}

        \item Suppose that $a=b<0$.
        \begin{enumerate}[label=(\theenumi.\arabic*), ref=\arabic*]     
        \item When $c=-a$, we have
        \[P_{a,a|-a} = \sum M_{a,a|-a} + \sum M_{a,a+1|-a-1}\]
        for $a<-\frac{1}{2}$, and
        \begin{align*}
            P_{-\frac{1}{2},-\frac{1}{2}|\frac{1}{2}}
            &= \sum M_{-\frac{1}{2},-\frac{1}{2}|\frac{1}{2}} + M_{\frac{1}{2},\frac{1}{2}|\frac{1}{2}} \\
            &+ M_{-\frac{1}{2},\frac{1}{2}|-\frac{1}{2}} + M_{\frac{1}{2},-\frac{1}{2}|-\frac{1}{2}} + M_{\frac{1}{2},\frac{1}{2}|-\frac{1}{2}} \\
            &+ \sum M_{-\frac{1}{2},\frac{3}{2}|\frac{3}{2}}.
        \end{align*}
        
        \item When $c=a$, we have
        \[P_{a,a|a} = \sum M_{a,a|a} + \sum M_{a,a+1|a+1}\]
        for $a<-\frac{1}{2}$, and
        \[P_{-\frac{1}{2},-\frac{1}{2}|-\frac{1}{2}} = \sum M_{-\frac{1}{2},-\frac{1}{2}|-\frac{1}{2}}.\]
        \end{enumerate}
    \end{enumerate}
\end{thm}

\subsection{Proof}
In this subsection, we prove Theorems \ref{thm:spo431} through \ref{thm:spo434}. We use the method of translation functors by effecting certain finite-dimensional representations. These representations, which are all irreducible, highest-weight, and self-dual (cf. \cite{cheng2012dualities}), and their weights are given in Table~\ref{tab:repsandweights}. All weights, except the zero weight, appear with multiplicity $1$. The zero weight is stated with its total multiplicity (i.e. $3 \cdot 0$ means the zero-weight space is three-dimensional).

\begin{table}[hp]
    \centering
    \[\begin{array}{c | c | c | c} 
        \toprule
        \text{Representation} & \text{Weights} & \text{Dimension} & \text{Highest Weight} \\
        \midrule
        
        V & \pm \{\delta_1, \delta_2, \epsilon\} \cup \{0\} & 7 & \delta_1 \\
        
        \midrule
        \multirow{2}{*}{$S^2 V$} & \pm \{\delta_1 \pm \delta_2, \delta_1 \pm \epsilon, \delta_1, \delta_2 \pm \epsilon, \delta_2, 2\epsilon, \epsilon\} & \multirow{2}{*}{24} & \multirow{2}{*}{$\delta_1 + \delta_2$} \\
         &  \cup \{ 4 \cdot 0\} & \\
         
        \midrule 
        \multirow{2}{*}{$\mathfrak{g}$} & \pm \{ 2\delta_1, \delta_1 \pm \delta_2,  \delta_1 \pm \epsilon, \delta_1, 2\delta_2,  \delta_2 \pm \epsilon, \delta_2, \epsilon
         \} & \multirow{2}{*}{25} & \multirow{2}{*}{$2\delta_1$} \\
         &  \cup \{3 \cdot 0\} &  \\
        \bottomrule
    \end{array}\]
    \caption{Representations and Weights}
    \label{tab:repsandweights}
\end{table}

In particular, we have that as $\osp(3|4)$-modules, $V \cong L_{3/2, -1/2 | 1/2} = L_{\delta_1 + \rho}$, $S^2 V \cong L_{3/2, 1/2 | 1/2} = L_{\delta_1 + \delta_2 + \rho}$, and $\g \cong L_{5/2, -1/2 | 1/2} = L_{2\delta_1 + \rho}$.
 
We now offer justification for the formulae above, separated into cases that have different formulae, based on the strategy in \cref{strategy}. Our proof will be more explicit in cases which require more sophisticated techniques; those which lack much explanation follow the strategy almost directly and list only the choices of $P_{\mu}$ and representation for translation functor. In particular, in the proofs we shall skip the following standard cases:
\begin{enumerate}
    \item When $\lambda=(a,b|\pm a)$, $\mu=(a+1,b|\pm a)$, $a \neq b$, and $P_{\lambda}=\mathrm{pr}_{\lambda}\left(P_{\mu} \otimes V\right)$ by Lemma~\ref{lowest} and Proposition~\ref{filprop} directly.
            
    \item When $\lambda=(a,b|\pm b)$, $\mu=(a,b+1|\pm b)$, $a \neq b$, and $P_{\lambda}=\mathrm{pr}_{\lambda}\left(P_{\mu} \otimes V\right)$ by Lemma~\ref{lowest} and Proposition~\ref{filprop} directly.
\end{enumerate}

\begin{proof}[Proof of Theorem~\ref{thm:spo431}]
Let $\lambda-\rho = (a,b \ | \ c)-\rho$ be an atypical weight with $a,b,c \in \frac{1}{2} + \mathbb{Z}$ and $a,b>0$.
    
    \begin{enumerate}[label=(\arabic*), ref=\arabic*]
        \item When $\lambda=(a,a-1|\pm (a-1))$, let $\mu=(a+1,a-1|\pm (a-1))$. The projective $P_{\mu}$ is a standard case. We use $V$ for translation functor. 
        
        \item When $\lambda=(a,a|\pm a)$, let $\mu=(a,a|\pm (a+1))$. We use $V$ for translation functor. 
    \end{enumerate}
\end{proof}

\begin{proof}[Proof of Theorem~\ref{thm:spo432}]
Let $\lambda-\rho = (a,b \ | \ c)-\rho$ be an atypical weight with $a,b,c \in \frac{1}{2} + \mathbb{Z}$ and $a>0>b$.

    \begin{enumerate}[label=(\arabic*), ref=\arabic*]
        \item When $\lambda=(a,-\frac{1}{2}|\frac{1}{2})$:
        \begin{enumerate}[label=(\theenumi.\arabic*), ref=\arabic*]
            \item If $a>\frac{3}{2}$,
            \[\mathrm{pr}_{\lambda}\left(P_{a,\frac{1}{2}|\frac{1}{2}} \otimes V\right) = M_{a,-\frac{1}{2}|\frac{1}{2}} + M_{a,\frac{1}{2}|\frac{1}{2}} + M_{a,\frac{1}{2}|-\frac{1}{2}} + M_{a,\frac{3}{2}|\frac{3}{2}}.\]
            By Lemma~\ref{lowest}, $P_{\lambda}$ must appear in the projection as a direct summand, and Proposition~\ref{filprop} ensures that the first three terms appear in $P_{\lambda}$. However, since $M_{a,\frac{3}{2}|\frac{3}{2}}$ does not form a projective on its own, it must also belong to $P_{\lambda}$.
            
            \item If $a=\frac{3}{2}$,
            \begin{align*}
                \mathrm{pr}_{\lambda}\left(P_{\frac{3}{2},\frac{1}{2}|\frac{1}{2}} \otimes V\right) 
                &= M_{\frac{3}{2},-\frac{1}{2}|\frac{1}{2}} 
                + M_{\frac{3}{2},\frac{1}{2}|\frac{1}{2}} 
                + M_{\frac{3}{2},\frac{1}{2}|-\frac{1}{2}} 
                + M_{\frac{3}{2},\frac{3}{2}|\frac{3}{2}} \\
                &+ M_{\frac{5}{2},\frac{3}{2}|\frac{5}{2}}.
            \end{align*}
            By Lemma~\ref{lowest} and Proposition~\ref{filprop}, the first three terms appear in $P_{\lambda}$. Since the standard filtration of $P_{\frac{3}{2},\frac{3}{2}|\frac{3}{2}}$ does not appear in the projection, $M_{\frac{3}{2},\frac{3}{2}|\frac{3}{2}}$ must belong to $P_{\lambda}$. Similarly, $M_{\frac{5}{2},\frac{3}{2}|\frac{5}{2}}$ must belong to $P_{\lambda}$.
        \end{enumerate}
            
        \item When $\lambda=(a,-\frac{1}{2}|-\frac{1}{2})$:
        \begin{enumerate}[label=(\theenumi.\arabic*), ref=\arabic*]
            \item If $a>\frac{3}{2}$,
            \begin{align*}
                \mathrm{pr}_{\lambda}\left(P_{a,-\frac{1}{2}|\frac{1}{2}} \otimes V\right)
                &= 2M_{a,-\frac{1}{2}|-\frac{1}{2}} + 2M_{a,-\frac{1}{2}|\frac{1}{2}} + 2M_{a,\frac{1}{2}|-\frac{1}{2}} + 3M_{a,\frac{1}{2}|\frac{1}{2}}\\
                &+ M_{a,\frac{3}{2}|\frac{3}{2}}.
            \end{align*}
            By Lemma~\ref{lowest}, $P_{\lambda}$ must appear twice in the projection as a direct summand, and By Proposition~\ref{filprop}, one copy of each of the first four terms must be in $P_{\lambda}$. Now, one copy of the fourth term and the last term remain. However, since only one copy of these two terms  remains, they cannot appear in $P_{\lambda}$. Thus, we get that 
            \[\mathrm{pr}_{\lambda}\left(P_{a,-\frac{1}{2}|\frac{1}{2}} \otimes V\right) = 2P_{a,-\frac{1}{2}|-\frac{1}{2}} + P_{a,\frac{1}{2}|\frac{1}{2}},\] and
            \[P_{a,-\frac{1}{2}|-\frac{1}{2}} = M_{a,-\frac{1}{2}|-\frac{1}{2}} + M_{a,-\frac{1}{2}|\frac{1}{2}} + M_{a,\frac{1}{2}|-\frac{1}{2}} + M_{a,\frac{1}{2}|\frac{1}{2}}.\]
                
            \item If $b=-\frac{1}{2}$ and $a=\frac{3}{2}$,
            we get a formula consistent with the previous case by applying the same method.
        \end{enumerate}
        
        \item When $\lambda=(a,-a|a)$: 
        \begin{enumerate}[label=(\theenumi.\arabic*), ref=\arabic*] 
            \item If $a>\frac{1}{2}$, let $\mu=(a+1,-a|a)$. The projective $P_{\mu}$ is a standard case. We use $V$ for translation functor. 
            
            \item If $a=\frac{1}{2}$,
            \begin{align*}
                \mathrm{pr}_{\lambda}\left(P_{\frac{3}{2},-\frac{1}{2}|\frac{1}{2}} \otimes V\right)
                &= M_{\frac{1}{2},-\frac{1}{2}|\frac{1}{2}} 
                + M_{\frac{1}{2},\frac{1}{2}|\frac{1}{2}} 
                + M_{\frac{1}{2},\frac{1}{2}|-\frac{1}{2}} 
                + M_{\frac{1}{2},\frac{3}{2}|\frac{3}{2}} \\
                &+ M_{\frac{3}{2},-\frac{1}{2}|\frac{3}{2}}
                + M_{\frac{3}{2},\frac{1}{2}|-\frac{3}{2}}
                + 2M_{\frac{3}{2},\frac{1}{2}|\frac{3}{2}}
                + M_{\frac{5}{2},\frac{1}{2}|\frac{5}{2}}.
            \end{align*}
            By Lemma~\ref{lowest}, $P_{\lambda}$ must appear in the projection. By Proposition~\ref{filprop}, the first six term and one copy of $M_{\frac{3}{2},\frac{1}{2}|\frac{3}{2}}$ must appear in $P_{\lambda}$. However, we run into a problem here: the two remaining terms, $M_{\frac{3}{2},\frac{1}{2}|\frac{3}{2}}$, $M_{\frac{5}{2},\frac{1}{2}|\frac{5}{2}}$ could actually form the projective $P_{\frac{3}{2},\frac{1}{2}|\frac{3}{2}}$. This means that we have to devise a different method to show they are also included in $P_{\lambda}$. 
                
            We have two possible standard filtrations of $P_{\lambda}$. Call the shorter one, which do not include the two unexplained terms, $Q$, and call $P_{\frac{3}{2},\frac{1}{2}|\frac{3}{2}}=M_{\frac{3}{2},\frac{1}{2}|\frac{3}{2}} + M_{\frac{5}{2},\frac{1}{2}|\frac{5}{2}}$ R. We shall show that $P_{\lambda}$ has the longer standard filtration, which we shall denote by abuse of notation $Q+R$, by proving that $Q$ is not a projective. We calculate the projections $\mathrm{pr}_{\mu}\left(Q \otimes \g\right)$ and $\mathrm{pr}_{\mu}\left(R \otimes \g\right)$ in Table~\ref{tab:1-11}.
                
            \begin{table}[hp]
                \centering
                \[\begin{array}{c | l | l l l}
                    \toprule 
                    \text{Projective} & \text{Terms} & \multicolumn{3}{c}{\mathrm{pr}_{\mu}(-\otimes \g)}\\
                    \midrule
                    \multirow{7}{*}{$Q$} 
                    & M_{\frac{1}{2},-\frac{1}{2}|\frac{1}{2}} & \color{red} M_{\frac{1}{2},-\frac{5}{2}|\frac{1}{2}} & \color{red} M_{\frac{5}{2},-\frac{1}{2}|\frac{1}{2}} & \\
                    & M_{\frac{1}{2},\frac{1}{2}|\frac{1}{2}} & \color{red} M_{\frac{1}{2},\frac{5}{2}|\frac{1}{2}} & \color{red} M_{\frac{5}{2},\frac{1}{2}|\frac{1}{2}} & \\
                    & M_{\frac{1}{2},\frac{1}{2}|-\frac{1}{2}} & \color{blue} M_{\frac{1}{2},\frac{5}{2}|-\frac{1}{2}} & \color{blue} M_{\frac{5}{2},\frac{1}{2}|-\frac{1}{2}} & \\
                    & M_{\frac{1}{2},\frac{3}{2}|\frac{3}{2}} & \color{blue} M_{\frac{1}{2},\frac{5}{2}|\frac{1}{2}} & \color{red} M_{\frac{3}{2},\frac{5}{2}|\frac{3}{2}} & \color{red} M_{\frac{5}{2},\frac{3}{2}|\frac{3}{2}} \\
                    & M_{\frac{3}{2},-\frac{1}{2}|\frac{3}{2}} & \color{red} M_{\frac{3}{2},-\frac{5}{2}|\frac{3}{2}} & \color{red} M_{\frac{5}{2},-\frac{3}{2}|\frac{3}{2}} & \color{violet} M_{\frac{5}{2},-\frac{1}{2}|\frac{1}{2}} \\
                    & M_{\frac{3}{2},\frac{1}{2}|-\frac{3}{2}} & \color{blue} M_{\frac{3}{2},\frac{5}{2}|-\frac{3}{2}} & \color{blue} M_{\frac{5}{2},\frac{3}{2}|-\frac{3}{2}} & \color{violet} M_{\frac{5}{2},\frac{1}{2}|-\frac{1}{2}} \\
                    & M_{\frac{3}{2},\frac{1}{2}|\frac{3}{2}} & \color{blue} M_{\frac{3}{2},\frac{5}{2}|\frac{3}{2}} & \color{blue} M_{\frac{5}{2},\frac{3}{2}|\frac{3}{2}} & \color{blue} M_{\frac{5}{2},\frac{1}{2}|\frac{1}{2}} \\
                    \midrule
                    \multirow{2}{*}{$R$} 
                    & M_{\frac{3}{2},\frac{1}{2}|\frac{3}{2}} & \color{brown} M_{\frac{3}{2},\frac{5}{2}|\frac{3}{2}} & \color{brown} M_{\frac{5}{2},\frac{3}{2}|\frac{3}{2}} & \color{violet} M_{\frac{5}{2},\frac{1}{2}|\frac{1}{2}} \\
                    & M_{\frac{5}{2},\frac{1}{2}|\frac{5}{2}} & \color{violet} M_{\frac{5}{2},\frac{3}{2}|\frac{3}{2}} & \color{brown} M_{\frac{5}{2},\frac{5}{2}|\frac{5}{2}} & \\
                    \bottomrule
                \end{array}\]
                \caption{Calculation of $P_{\frac{1}{2},-\frac{1}{2}|\frac{1}{2}}$}
                \label{tab:1-11}
            \end{table}
                
            If $Q$ were a projective, then $\mathrm{pr}_{\mu}\left(Q \otimes \g\right)$ is again a projective and by Lemma~\ref{lowest} must split into indecomposable projectives. We see that the lowest weight appearing is $\left(\frac{1}{2},-\frac{5}{2}|\frac{1}{2}\right)$, so $P_{\frac{1}{2},-\frac{5}{2}|\frac{1}{2}}$ must appear, and its terms are colored red. Next, we must have $P_{\frac{1}{2},\frac{5}{2}|-\frac{1}{2}}$ appear, whose terms are colored blue. Then, as $\left(\frac{5}{2},-\frac{1}{2}|\frac{1}{2}\right)$ is the next lowest weight, $P_{\frac{5}{2},-\frac{1}{2}|\frac{1}{2}}$ must appear (colored violet). However, we see that there are not enough terms left in $\mathrm{pr}_{\mu}(Q \otimes \g)$. Thus, $Q$ is not a projective and we must have $P_{\lambda} = Q + R$. It turns out that 
            \[\mathrm{pr}_{\mu}\left((Q+R) \otimes \g\right) = P_{\frac{1}{2},-\frac{5}{2}|\frac{1}{2}} + P_{\frac{1}{2},\frac{5}{2}|-\frac{1}{2}} + P_{\frac{5}{2},-\frac{1}{2}|\frac{1}{2}} + P_{\frac{3}{2},\frac{5}{2}|\frac{3}{2}}.\]
        \end{enumerate}
        
        \item When $\lambda = (a,-a|-a)$, let $\mu = (a+1,-a|-a)$. We use $V$ for translation functor. 
    \end{enumerate}
\end{proof}

\begin{proof}[Proof of Theorem~\ref{thm:spo433}]
Let $\lambda-\rho = (a,b \ | \ c)-\rho$ be an atypical weight with $a,b,c \in \frac{1}{2} + \mathbb{Z}$ and $b>0>a$.

    \begin{enumerate}[label=(\arabic*), ref=\arabic*]        
        \item When $\lambda = (-\frac{1}{2},b|\frac{1}{2})$:
        \begin{enumerate}[label=(\theenumi.\arabic*), ref=\arabic*]     
            \item $b>\frac{3}{2}$,
            \begin{align*}
                \mathrm{pr}_{\lambda}\left(P_{\frac{1}{2},b|-\frac{1}{2}} \otimes V\right)
                &= \sum M_{-\frac{1}{2}, b|\frac{1}{2}}\\
                &+ M_{\frac{1}{2}, b|-\frac{1}{2}} + M_{b, \frac{1}{2}|-\frac{1}{2}} + M_{\frac{3}{2}, b|\frac{3}{2}} + M_{b, \frac{3}{2}|\frac{3}{2}}.
            \end{align*}
            By Lemma~\ref{lowest} and Proposition~\ref{filprop}, $\sum M_{-\frac{1}{2}, b|\frac{1}{2}}$ must belong to $P_{\lambda}$. The lowest remaining term is $M_{\frac{1}{2}, b|-\frac{1}{2}}$, and since the remaining terms do not contain the standard filtration of $P_{\frac{1}{2}, b|-\frac{1}{2}}$, it must also belong to $P_{\lambda}$. By the same argument, each of the remaining terms belongs to $P_{\lambda}$.
                
            \item If $b=\frac{3}{2}$,
            \begin{align*}
                \mathrm{pr}_{\lambda}\left(P_{\frac{1}{2},\frac{3}{2}|-\frac{1}{2}} \otimes V\right)
                &= \sum M_{-\frac{1}{2}, \frac{3}{2}|\frac{1}{2}}\\
                &+ M_{\frac{1}{2}, \frac{3}{2}|-\frac{1}{2}} + M_{\frac{3}{2}, \frac{1}{2}|-\frac{1}{2}} + M_{\frac{3}{2}, \frac{3}{2}|\frac{3}{2}}.
            \end{align*}
            By similar arguments as above, the projection is equal to $P_{\lambda}$.
        \end{enumerate}
            
        \item When $\lambda = (-\frac{1}{2},b|-\frac{1}{2})$:
        \begin{enumerate}[label=(\theenumi.\arabic*), ref=\arabic*]
            \item If $b>\frac{3}{2}$,
            \begin{align*}
                \mathrm{pr}_{\lambda}\left(P_{-\frac{1}{2},b|\frac{1}{2}} \otimes V\right) 
                &= 2 \sum M_{-\frac{1}{2}, b|-\frac{1}{2}} \\
                &+ M_{\frac{1}{2}, b|\frac{1}{2}} + M_{b, \frac{1}{2}|\frac{1}{2}} + M_{\frac{3}{2}, b|\frac{3}{2}} + M_{b, \frac{3}{2}|\frac{3}{2}}.
            \end{align*}
            By Lemma~\ref{lowest}, two copies of $P_{\lambda}$ must appear in the projection. By Proposition~\ref{filprop}, $\sum M_{-\frac{1}{2}, b|-\frac{1}{2}}$ belongs to $P_{\lambda}$. Now, since the remaining four terms each only appear with multiplicity one, they cannot belong to $P_{\lambda}$. They form $P_{\frac{1}{2}, b|\frac{1}{2}}$.
                
            \item If $b=\frac{3}{2}$, we get the same result with the same projection as above, except that the projection has three (instead of four) remaining terms after subtracting two copies of $\sum M_{-\frac{1}{2}, b|-\frac{1}{2}}$. The three terms still form $P_{\frac{1}{2}, b|\frac{1}{2}}$, which has three instead of four terms when $b=\frac{3}{2}$. 
        \end{enumerate}

        \item When $\lambda = (a,-a|-a)$:
        \begin{enumerate}[label=(\theenumi.\arabic*), ref=\arabic*]     
            \item If $a<-\frac{1}{2}$,
            \begin{align*}
                \mathrm{pr}_{\lambda}\left(P_{a+1,-a|-a} \otimes V\right)
                &= \sum M_{a,-a|-a} + M_{-a,-a|-a} \\
                &+ \sum M_{a,-a+1|-a+1} + \sum M_{a+1,-a|-a-1}.
            \end{align*}
            By Lemma~\ref{lowest} and Proposition~\ref{filprop}, one copy of each term must appear in $P_{\lambda}$. Now, there remains only the second copy of the term $M_{-a,-a|-a}$, and as it clearly cannot form a projective, it must also belong to $P_{\lambda}$.
                
            \item If $a=-\frac{1}{2}$,
            \begin{align*}
                \mathrm{pr}_{\lambda}\left(P_{\frac{1}{2},\frac{1}{2}|\frac{1}{2}} \otimes V\right)
                &=\sum M_{-\frac{1}{2},\frac{1}{2}|\frac{1}{2}} + M_{\frac{1}{2},\frac{1}{2}|-\frac{1}{2}} + \sum M_{-\frac{1}{2},\frac{3}{2}|\frac{3}{2}}.
            \end{align*}
            By Lemma~\ref{lowest} and Proposition~\ref{filprop}, every term except for $M_{\frac{1}{2},\frac{1}{2}|-\frac{1}{2}}$ must appear in $P_{\lambda}$. As the one remaining term cannot form a projective, it must also belong to $P_{\lambda}$. Notice that unlike the previous term, each term here only appears with multiplicity one. 
        \end{enumerate}
            
        \item When $\lambda = (a,-a|a)$:
        \begin{enumerate}[label=(\theenumi.\arabic*), ref=\arabic*]
            \item If $a<-\frac{1}{2}$,
            \begin{align*}
                \mathrm{pr}_{\lambda}\left(P_{a+1,-a|a} \otimes V\right)
                &= \sum M_{a,-a|a} + M_{-a,-a|a} + M_{-a,-a|-a} \\
                &+ \sum M_{a,-a+1|a-1} + \sum M_{a+1,-a|a+1}.
            \end{align*}
            By Lemma~\ref{lowest} and Proposition~\ref{filprop}, one copy of each term must appear in $P_{\lambda}$. Now, there remain only the second copies of the terms $M_{-a,-a|a}$ and $M_{-a,-a|-a}$, and as they cannot form a projective, they must also belong to $P_{\lambda}$.
                
            \item If $a=-\frac{1}{2}$,
            \begin{align*}
                \mathrm{pr}_{\lambda}\left(P_{\frac{1}{2},\frac{1}{2}|-\frac{1}{2}} \otimes V\right)
                &= \sum M_{-\frac{1}{2},\frac{1}{2}|-\frac{1}{2}} + M_{\frac{1}{2},\frac{1}{2}|-\frac{1}{2}} + M_{\frac{1}{2},\frac{1}{2}|\frac{1}{2}} \\
                &+ \sum M_{-\frac{1}{2},\frac{3}{2}|-\frac{3}{2}}.
            \end{align*}
            By similar arguments as above, the projection is equal to $P_{\lambda}$.
        \end{enumerate}
    \end{enumerate}
\end{proof}

\begin{proof}[Proof of Theorem~\ref{thm:spo434}]
Let $\lambda-\rho = (a,b \ | \ c)-\rho$ be an atypical weight with $a,b,c \in \frac{1}{2} + \mathbb{Z}$ and $a,b<0$.

    \begin{enumerate}[label=(\arabic*), ref=\arabic*]
        \item When $\lambda = (a,-\frac{1}{2}|\frac{1}{2})$:
        \begin{enumerate}[label=(\theenumi.\arabic*), ref=\arabic*]
            \item If $a=-\frac{3}{2}$, we project $P_{-\frac{3}{2},\frac{1}{2}|\frac{1}{2}} \otimes V$ onto the $\lambda$ linkage block in Table~\ref{tab:pr-311}. Notice that in the table, to save space, we use the $\pm$ sign to combine two terms into one. For example, $M_{\frac{1}{2},\pm\frac{3}{2}|\frac{1}{2}}$ represents the two terms $M_{\frac{1}{2},\frac{3}{2}|\frac{1}{2}}$ and  $M_{\frac{1}{2},-\frac{3}{2}|\frac{1}{2}}$.
                    
            \begin{table}[hp]
                \centering
                \[\begin{array}{l | l l l l}                    \toprule
                    P_{-\frac{3}{2},\frac{1}{2}|\frac{1}{2}} & \multicolumn{4}{c}{\mathrm{pr}_{\lambda}\left(P_{-\frac{3}{2},\frac{1}{2}|\frac{1}{2}} \otimes V\right)} \\
                    \midrule
                    M_{-\frac{3}{2},\frac{1}{2}|\frac{1}{2}} & \color{red} M_{-\frac{3}{2},-\frac{1}{2}|\frac{1}{2}} & \color{red} M_{-\frac{3}{2},\frac{1}{2}|-\frac{1}{2}} & \color{red} M_{-\frac{3}{2},\frac{1}{2}|\frac{1}{2}} & \\
                            
                    M_{-\frac{1}{2},\frac{3}{2}|\frac{1}{2}} & \color{red} M_{-\frac{1}{2},\frac{3}{2}|-\frac{1}{2}} & M_{-\frac{1}{2},\frac{3}{2}|\frac{1}{2}} & \color{blue} M_{\frac{1}{2},\frac{3}{2}|\frac{1}{2}} & \\
                            
                    M_{\frac{1}{2},\pm\frac{3}{2}|\frac{1}{2}} & \color{red} M_{-\frac{1}{2},\pm\frac{3}{2}|\frac{1}{2}} & \color{red} M_{\frac{1}{2},\pm\frac{3}{2}|-\frac{1}{2}} & \color{red} M_{\frac{1}{2},\pm\frac{3}{2}|\frac{1}{2}} & \\
                            
                    M_{\frac{3}{2},\pm\frac{1}{2}|\frac{1}{2}} & \color{red} M_{\frac{3}{2},\pm\frac{1}{2}|\frac{1}{2}} & \color{blue} M_{\frac{3}{2},\frac{1}{2}|\frac{1}{2}} & M_{\frac{3}{2},-\frac{1}{2}|\frac{1}{2}} & \color{red} M_{\frac{3}{2},\pm\frac{1}{2}|-\frac{1}{2}} \\
                            
                    M_{-\frac{3}{2},\frac{3}{2}|\frac{3}{2}} & M_{-\frac{3}{2},\frac{3}{2}|\frac{3}{2}} & & & \\
                            
                    M_{\frac{3}{2},-\frac{3}{2}|\frac{3}{2}} & M_{\frac{3}{2},-\frac{3}{2}|\frac{3}{2}} & & & \\
                            
                    M_{\frac{3}{2},\frac{3}{2}|\frac{3}{2}} & \color{blue} M_{\frac{3}{2},\frac{3}{2}|\frac{3}{2}} & & & \\
                    \bottomrule
                \end{array}\]
                \caption{Calculation of $\mathrm{pr}_{\lambda}\left(P_{-\frac{3}{2},\frac{1}{2}|\frac{1}{2}} \otimes V\right)$}
                \label{tab:pr-311}
            \end{table}
                    
            We start by finding the terms that must appear in $P_{\lambda}$. By Proposition~\ref{filprop}, the terms colored red belong to $P_{\lambda}$. Now, the lowest remaining term is $M_{-\frac{3}{2},\frac{3}{2}|\frac{3}{2}}$, and since $P_{-\frac{3}{2},\frac{3}{2}|\frac{3}{2}}$ does not appear in the projection, it must belong to $P_{\lambda}$. By the same reasoning, the second copy of $M_{-\frac{1}{2},\frac{3}{2}|\frac{1}{2}}$ must also appear in $P_{\lambda}$. Now, the next lowest term is $M_{\frac{1}{2},\frac{3}{2}|\frac{1}{2}}$. However, the three terms in the standard filtration of $P_{\frac{1}{2},\frac{3}{2}|\frac{1}{2}}$ all remain in the projection, colored blue. The two remaining unsorted terms must belong to $P_{\lambda}$ as no more projective can form among them. Again, we face two possible standard filtrations of $P_{\lambda}$. Denote the one containing all red and black terms $Q$, and denote the three blue terms $R$. We shall show that $Q$ is not a projective and thereby prove that $P_{\lambda} = Q + R$. 
                    
            Let $\mu = \left( -\frac{3}{2},-\frac{1}{2}|\frac{3}{2} \right)$. We project $Q \otimes V$ and $R \otimes V$ onto the $\mu$ linkage block in Table~\ref{tab:-3-11}. In the table we use our notation of $\sum M_{\lambda}$ to combine terms.
                    
            \begin{table}[hp]
                \centering
                \[\begin{array}{c | l | l l}
                    \toprule 
                    \text{Projective} & \text{Terms} & \multicolumn{2}{c}{\mathrm{pr}_{\mu}(-\otimes V)}\\
                    \midrule
                    \multirow{9}{*}{$Q$} 
                            
                    & \multirow{2}{*}{$\sum M_{-\frac{3}{2},-\frac{1}{2}|\frac{1}{2}}$} & \color{red}\sum M_{-\frac{3}{2},-\frac{1}{2}|\frac{3}{2}} & \color{red}\sum M_{-\frac{1}{2},-\frac{1}{2}|\frac{1}{2}}\\
                    & & \color{violet} \sum M_{-\frac{1}{2},-\frac{1}{2}|\frac{1}{2}} &\\
                            
                    & M_{-\frac{3}{2},\frac{1}{2}|-\frac{1}{2}} & \color{blue} M_{-\frac{3}{2},\frac{1}{2}|-\frac{3}{2}} & \color{blue} M_{-\frac{1}{2},\frac{1}{2}|-\frac{1}{2}} \\
                            
                    & M_{-\frac{1}{2},\frac{3}{2}|-\frac{1}{2}} & \color{blue} M_{-\frac{1}{2},\frac{3}{2}|-\frac{3}{2}} & \color{violet} M_{-\frac{1}{2},\frac{1}{2}|-\frac{1}{2}} \\
                            
                    & M_{\frac{1}{2},\pm\frac{3}{2}|-\frac{1}{2}} & \color{blue} M_{\frac{1}{2},\pm\frac{3}{2}|-\frac{3}{2}} & \color{blue} M_{\frac{1}{2},\pm\frac{1}{2}|-\frac{1}{2}} \\
                            
                    & M_{\frac{3}{2},\pm\frac{1}{2}|-\frac{1}{2}} & \color{blue} M_{\frac{3}{2},\pm\frac{1}{2}|-\frac{3}{2}} & \color{violet} M_{\frac{1}{2},\pm\frac{1}{2}|-\frac{1}{2}} \\
                            
                    & M_{-\frac{1}{2},\frac{3}{2}|\frac{1}{2}} & \color{blue} M_{-\frac{1}{2},\frac{3}{2}|\frac{3}{2}} & \color{blue} M_{-\frac{1}{2},\frac{1}{2}|\frac{1}{2}} \\
                            
                    & M_{\frac{3}{2},-\frac{1}{2}|\frac{1}{2}} & \color{blue} M_{\frac{3}{2},-\frac{1}{2}|\frac{3}{2}} & \color{blue} M_{\frac{1}{2},-\frac{1}{2}|\frac{1}{2}} \\
                            
                    & M_{-\frac{3}{2},\frac{3}{2}|\frac{3}{2}} & \color{blue} M_{-\frac{3}{2},\frac{1}{2}|\frac{3}{2}} & \color{brown} M_{-\frac{1}{2},\frac{3}{2}|\frac{3}{2}} \\
                            
                    & M_{\frac{3}{2},-\frac{3}{2}|\frac{3}{2}} & \color{blue} M_{\frac{1}{2},-\frac{3}{2}|\frac{3}{2}} & \color{brown} M_{\frac{3}{2},-\frac{1}{2}|\frac{3}{2}} \\
                    \midrule
                            
                    \multirow{3}{*}{$R$} 
                            
                    & M_{\frac{1}{2},\frac{3}{2}|\frac{1}{2}} & \color{blue} M_{\frac{1}{2},\frac{3}{2}|\frac{3}{2}} & \color{blue} M_{\frac{1}{2},\frac{1}{2}|\frac{1}{2}} \\
                            
                    & M_{\frac{3}{2},\frac{1}{2}|\frac{1}{2}} & \color{blue} M_{\frac{3}{2},\frac{1}{2}|\frac{3}{2}} & M_{\frac{1}{2},\frac{1}{2}|\frac{1}{2}} \\
                            
                    & M_{\frac{3}{2},\frac{3}{2}|\frac{3}{2}} & M_{\frac{1}{2},\frac{3}{2}|\frac{3}{2}} & M_{\frac{3}{2},\frac{1}{2}|\frac{3}{2}} \\
                    \bottomrule
                \end{array}\]
                \caption{Calculation of $P_{-\frac{3}{2},-\frac{1}{2}|\frac{1}{2}}$}
                \label{tab:-3-11}
            \end{table}
                    
            We find indecomposable projectives in this projection starting with the lowest term. First, $P_{\mu}$ appears, colored red. The next lowest is $M_{-\frac{3}{2},\frac{1}{2}|-\frac{3}{2}}$, so $P_{-\frac{3}{2},\frac{1}{2}|-\frac{3}{2}}$ must appear, colored blue. Since $Q$ does not have enough terms for this, it cannot be a projective and we must have that $P_{\lambda} = Q + R$. 
                    
            \item Here we calculate the projective $P_{\nu}$ where $\nu=\left(-\frac{1}{2},-\frac{1}{2}|\frac{1}{2}\right)$.
            If we keep isolating projectives from the projection above, we find that the next projective that must appear is $P_{-\frac{1}{2},-\frac{1}{2}|\frac{1}{2}}$. By Proposition~\ref{filprop}, the terms colored violet must appear. Of the five remaining terms, the two in the projection (colored brown) of $Q$ must also belong to $P_{-\frac{1}{2},-\frac{1}{2}|\frac{1}{2}}$ as they cannot belong to another projective, while the three in the projection of $R$ could form the projective $T=P_{\frac{1}{2},\frac{1}{2}|\frac{1}{2}}$. Thus, we have two possible standard filtrations of $P_{-\frac{1}{2},-\frac{1}{2}|\frac{1}{2}}$, which we denote by $S$ and $S+T$. Now, we project $S \otimes V$ and $T \otimes V$ back onto the linkage block of $\lambda = \left(-\frac{3}{2},-\frac{1}{2}|\frac{1}{2}\right)$ in Table~\ref{tab:-1-11}.
                    
            \begin{table}
                \centering
                \[\begin{array}{c | l | l l}
                    \toprule 
                    \text{Projective} & \text{Terms} & \multicolumn{2}{c}{\mathrm{pr}_{\lambda}(-\otimes V)}\\
                    \midrule
                    \multirow{5}{*}{$S$} 
                            
                    & \sum M_{-\frac{1}{2},-\frac{1}{2}|\frac{1}{2}} & \color{red}\sum M_{-\frac{3}{2},-\frac{1}{2}|\frac{1}{2}} & \\
                            
                    & M_{-\frac{1}{2},\frac{1}{2}|-\frac{1}{2}} & \color{red}M_{-\frac{3}{2},\frac{1}{2}|-\frac{1}{2}} & \color{red}M_{-\frac{1}{2},\frac{3}{2}|-\frac{1}{2}} \\
                            
                    & M_{\frac{1}{2},\pm\frac{1}{2}|-\frac{1}{2}} & \color{red}M_{\frac{1}{2},\pm\frac{3}{2}|-\frac{1}{2}} & \color{red}M_{\frac{3}{2},\pm\frac{1}{2}|-\frac{1}{2}} \\
                            
                    & M_{-\frac{1}{2},\frac{3}{2}|\frac{3}{2}} & \color{red}M_{-\frac{3}{2},\frac{3}{2}|\frac{3}{2}} & \color{red}M_{-\frac{1}{2},\frac{3}{2}|\frac{1}{2}}\\
                            
                    & M_{\frac{3}{2},-\frac{1}{2}|\frac{3}{2}} & \color{red}M_{\frac{3}{2},-\frac{3}{2}|\frac{3}{2}} & \color{red}M_{\frac{3}{2},-\frac{1}{2}|\frac{1}{2}}\\
                            
                    \midrule
                            
                    \multirow{3}{*}{$T$} 
                    & M_{\frac{1}{2},\frac{1}{2}|\frac{1}{2}} & \color{red}M_{\frac{1}{2},\frac{3}{2}|\frac{1}{2}} & \color{red}M_{\frac{3}{2},\frac{1}{2}|\frac{1}{2}}\\
                            
                    & M_{\frac{1}{2},\frac{3}{2}|\frac{3}{2}} & M_{\frac{1}{2},\frac{3}{2}|\frac{1}{2}} & \color{red}M_{\frac{3}{2},\frac{3}{2}|\frac{3}{2}}\\
                            
                    & M_{\frac{3}{2},\frac{1}{2}|\frac{3}{2}} & M_{\frac{3}{2},\frac{1}{2}|\frac{1}{2}} & M_{\frac{3}{2},\frac{3}{2}|\frac{3}{2}}\\
                    \bottomrule
                \end{array}\]
                \caption{Calculation of $P_{-\frac{1}{2},-\frac{1}{2}|\frac{1}{2}}$}
                \label{tab:-1-11}
            \end{table}
                    
            By Lemma~\ref{lowest}, $P_{\lambda}$ must appear in the projection, and we color its terms red. We see that $S$ does not have enough terms and thus cannot be a projective. Thus, $P_{\nu} = S+T$, and 
            \[\mathrm{pr}_{\lambda}((S+T)\otimes V) = P_{-\frac{3}{2},-\frac{1}{2}|\frac{1}{2}} + P_{\frac{1}{2},\frac{3}{2}|\frac{1}{2}}.\]

            \item If $a<-\frac{3}{2}$, we project $P_{a,\frac{1}{2}|\frac{1}{2}} \otimes V$ onto the $\lambda$ linkage block. Similar to the previous case, we obtain two possible standard filtrations, denoted $Q(a)$ and $(Q+R)(a)$. Now, we consider the specific case of $a=-\frac{5}{2}$, and we project the corresponding $Q(-\frac{5}{2}) \otimes V$ and $R(-\frac{5}{2}) \otimes V$ onto the $\left(-\frac{3}{2},-\frac{1}{2}|\frac{1}{2}\right)$ block. It turns out that $Q(-\frac{5}{2})$ is not a projective and
            \[\mathrm{pr}_{-\frac{3}{2},-\frac{1}{2},\frac{1}{2}}((Q+R) \otimes V) = P_{-\frac{5}{2},\frac{3}{2}|\frac{5}{2}} + P_{-\frac{3}{2},-\frac{1}{2}|\frac{1}{2}}.\]
            Thus, $P_{-\frac{5}{2},\frac{1}{2}|\frac{1}{2}} = (Q+R)(-\frac{5}{2})$. Then, we proceed by induction, projecting $Q(a-1) \otimes V$ and $R(a-1) \otimes V$ onto the $\left(a,\frac{1}{2}|\frac{1}{2}\right)$ block. Since we find that the projection of $(Q+R)(a-1) \otimes V$ is equal to $P_{a,\frac{1}{2}|\frac{1}{2}} = (Q+R)(a)$, $Q(a-1)$ does not have enough terms and thus is not a projective. This way, we show that $P_{a,\frac{1}{2}|\frac{1}{2}} = (Q+R)(a)$ for all $a<-\frac{3}{2}$.
        \end{enumerate}

        \item When $\lambda=(a,a|-a)$:
        \begin{enumerate}[label=(\theenumi.\arabic*), ref=\arabic*] 
            \item If $a<-\frac{1}{2}$,
            \[\mathrm{pr}_{\lambda}\left(P_{a+1,a+1|-a} \otimes S^2 V \right) = \sum M_{a,a|-a} + \sum M_{a,a+1|-a-1}.\]
            By Lemma~\ref{lowest} and Proposition~\ref{filprop}, the projection is equal to $P_{\lambda}$.
                
            \item The case $\lambda=\left(-\frac{1}{2},-\frac{1}{2}|\frac{1}{2}\right)$ was resolved above.
        \end{enumerate}
        
        \item When $\lambda=(a,a|a)$:
        \begin{enumerate}[label=(\theenumi.\arabic*), ref=\arabic*]  
            \item If $a<-\frac{1}{2}$,
            \[\mathrm{pr}_{\lambda}\left(P_{a+1,a+1|a} \otimes S^2 V \right) = \sum M_{a,a|a} + \sum M_{a,a+1|a+1}.\]
            By Lemma~\ref{lowest} and Proposition~\ref{filprop}, the projection is equal to $P_{\lambda}$.
                
            \item If $a=-\frac{1}{2}$,
            \begin{align*}
                \mathrm{pr}_{\lambda}\left(P_{-\frac{1}{2},-\frac{1}{2}|\frac{1}{2}} \otimes V \right) 
                &= 3 \sum M_{-\frac{1}{2},-\frac{1}{2}|-\frac{1}{2}} \\
                &+ 2 \left(\sum M_{-\frac{1}{2},\frac{1}{2}|\frac{1}{2}} + M_{\frac{1}{2},\frac{1}{2}|-\frac{1}{2}} + \sum M_{-\frac{1}{2},\frac{3}{2}|\frac{3}{2}} \right).
            \end{align*}
            Since the lowest term is $M_{-\frac{1}{2},-\frac{1}{2}|-\frac{1}{2}}$ and it appears with multiplicity $3$, by Lemme~\ref{lowest}, $P_{\lambda}$ must appear three times in the projection. By Proposition~\ref{filprop}, $\sum M_{-\frac{1}{2},-\frac{1}{2}|-\frac{1}{2}}$ belongs to $P_{\lambda}$. Since these are also the only terms with multiplicity at least $3$, 
            \[P_{\lambda} = \sum M_{-\frac{1}{2},-\frac{1}{2}|-\frac{1}{2}}.\]
            It turns out that
            \[\mathrm{pr}_{\lambda}\left(P_{-\frac{1}{2},-\frac{1}{2}|\frac{1}{2}} \otimes V \right) = 3P_{-\frac{1}{2},-\frac{1}{2}|-\frac{1}{2}} + 2P_{-\frac{1}{2},\frac{1}{2}|\frac{1}{2}}.\]
        \end{enumerate}
        
        \item When $\lambda = \left(a,-\frac{1}{2}|-\frac{1}{2}\right)$, we first consider the base case where $a=-\frac{3}{2}$. We have that
        \[\mathrm{pr}_{\left(-\frac{3}{2},-\frac{1}{2}|-\frac{1}{2}\right)}\left(P_{-\frac{1}{2},-\frac{1}{2}|-\frac{1}{2}} \otimes V \right) = \sum M_{-\frac{3}{2},-\frac{1}{2}|-\frac{1}{2}}.\]
        By Lemma~\ref{lowest} and Proposition~\ref{filprop}, the projection is equal to $P_{-\frac{3}{2},-\frac{1}{2}|-\frac{1}{2}}$. Now, we proceed by induction,
        \[\mathrm{pr}_{\left(a,-\frac{1}{2}|-\frac{1}{2}\right)}\left(P_{a+1,-\frac{1}{2}|-\frac{1}{2}} \otimes V \right) = \sum M_{a,-\frac{1}{2}|-\frac{1}{2}},\]
        and get By Proposition~\ref{filprop} that 
        \[P_{a,-\frac{1}{2}|-\frac{1}{2}} = \sum M_{a,-\frac{1}{2}|-\frac{1}{2}}\]
        for all $a<-\frac{1}{2}$ (in fact, for $a=-\frac{1}{2}$ as well, as shown in the previous subcase).

        \item When $\lambda = \left(-\frac{1}{2},b|\frac{1}{2}\right)$
        \begin{enumerate}[label=(\theenumi.\arabic*), ref=\arabic*]     
            \item If $b=-\frac{3}{2}$, we have 
            \begin{align*}
                \mathrm{pr}_{\lambda}\left(P_{\frac{1}{2},-\frac{3}{2}|\frac{1}{2}} \otimes V\right)
                &= \sum M_{-\frac{1}{2},-\frac{3}{2}|\frac{1}{2}} + \color{red}M_{\frac{3}{2},-\frac{1}{2}|\frac{1}{2}} + \color{red}M_{\frac{3}{2},\frac{1}{2}|\frac{1}{2}} \\
                &+ M_{\frac{1}{2},-\frac{3}{2}|-\frac{1}{2}} + M_{\frac{1}{2},\frac{3}{2}|-\frac{1}{2}} + M_{\frac{3}{2}, -\frac{1}{2}|-\frac{1}{2}} + M_{\frac{3}{2}, \frac{1}{2}|-\frac{1}{2}} \\
                &+ \color{red}M_{\frac{3}{2},-\frac{3}{2}|\frac{3}{2}} + \color{red}M_{\frac{3}{2},\frac{3}{2}|\frac{3}{2}}.
            \end{align*}
            By Proposition~\ref{filprop}, all terms except for the four colored red must appear in $P_{\lambda}$. Proceeding from the lowest remaining term, we can show that each of the remaining term must appear in $P_{\lambda}$ since no other projective can form in the projection.
                
            \item If $b<-\frac{3}{2}$, we have
            \begin{align*}
                \mathrm{pr}_{\lambda}\left(P_{\frac{1}{2},b|\frac{1}{2}} \otimes V\right)
                &= \sum M_{-\frac{1}{2},b|\frac{1}{2}} + \color{red}M_{-b,-\frac{1}{2}|\frac{1}{2}} + \color{red}M_{-b,\frac{1}{2}|\frac{1}{2}} \\
                &+ M_{\frac{1}{2},b|-\frac{1}{2}} + M_{\frac{1}{2},-b|-\frac{1}{2}} + M_{-b, -\frac{1}{2}|-\frac{1}{2}} + M_{-b, \frac{1}{2}|-\frac{1}{2}} \\
                &+ M_{\frac{3}{2},b|\frac{3}{2}} + M_{\frac{3}{2},-b|\frac{3}{2}} + \color{red}M_{-b,-\frac{3}{2}|\frac{3}{2}} + \color{red}M_{-b,\frac{3}{2}|\frac{3}{2}}.
            \end{align*}
            By a similar argument as above, we can show that all terms except for the four terms colored red, which can form the projective $P_{-b,-\frac{3}{2}|\frac{3}{2}}$, must appear in $P_{\lambda}$. We now have two possible standard filtrations for $P_{\lambda}$. As usual, call them $Q(b)$ and $(Q+R)(b)$. First we consider the case $b=-\frac{5}{2}$. We project $Q(-\frac{5}{2}) \otimes V$ and $R(-\frac{5}{2}) \otimes V$ back to the $\left( -\frac{1}{2},-\frac{3}{2}|\frac{1}{2} \right)$ block in Table~\ref{tab:-1-51}.
                
            \begin{table}
                \centering
                \[\begin{array}{c | l | l l}
                    \toprule 
                    \text{Projective} & \text{Terms} & \multicolumn{2}{c}{\mathrm{pr}_{-\frac{1}{2},-\frac{3}{2}|\frac{1}{2}}(-\otimes V)}\\
                    \midrule
                    \multirow{4}{*}{$Q(-\frac{5}{2})$} 
                        
                    & \sum M_{-\frac{1}{2},-\frac{5}{2}|\frac{1}{2}} & \color{red}\sum M_{-\frac{1}{2},-\frac{3}{2}|\frac{1}{2}} & \\
                        
                    & M_{\frac{1}{2},\pm \frac{5}{2}|-\frac{1}{2}} & \color{red}M_{\frac{1}{2},\pm \frac{3}{2}|-\frac{1}{2}} & \\
                        
                    & M_{\frac{5}{2},\pm \frac{1}{2}|-\frac{1}{2}} & \color{red}M_{\frac{3}{2},\pm \frac{1}{2}|-\frac{1}{2}} & \\
                        
                    & M_{\frac{3}{2},\pm \frac{5}{2}|\frac{3}{2}} & \color{red}M_{\frac{3}{2},\pm \frac{3}{2}|\frac{3}{2}} & M_{\frac{3}{2},\pm \frac{5}{2}|\frac{5}{2}} \\
                        
                    \midrule
                        
                    \multirow{2}{*}{$R(-\frac{5}{2})$} 
                    & M_{\frac{5}{2},\pm \frac{3}{2}|\frac{3}{2}} & M_{\frac{3}{2},\pm \frac{3}{2}|\frac{3}{2}} & M_{\frac{5}{2},\pm \frac{3}{2}|\frac{5}{2}} \\
                    & M_{\frac{5}{2},\pm \frac{1}{2}|\frac{1}{2}} & \color{red}M_{\frac{3}{2},\pm \frac{1}{2}|\frac{1}{2}} & \\
                    \bottomrule
                \end{array}\]
                \caption{Calculation of $P_{-\frac{1}{2},-\frac{5}{2}|\frac{1}{2}}$}
                \label{tab:-1-51}
            \end{table}
                
            Since the lowest term appearing in the projection is $M_{-\frac{1}{2},-\frac{3}{2}|\frac{1}{2}}$, the projective $P_{-\frac{1}{2},-\frac{3}{2}|\frac{1}{2}}$ must appear in the projection, with its terms colored red. We see that $Q$ again does not have enough terms and is therefore not a projective. Thus, $P_{-\frac{1}{2},-\frac{5}{2}|\frac{1}{2}} = (Q+R)(-\frac{5}{2})$. As we have the base case now, we may proceed by induction. By projecting $Q(b-1) \otimes V$ and $R(b-1) \otimes V$ onto the $\left(-\frac{1}{2},b|\frac{1}{2}\right)$ block, we see that $Q(b-1)$ does not have enough terms and
            \[\mathrm{pr}_{-\frac{1}{2},b|\frac{1}{2}}((Q+R)(b-1) \otimes V) = (Q+R)(b).\]
            Thus, for all $b<-\frac{3}{2}$, $P_{-\frac{1}{2},b|\frac{1}{2}} = (Q+R)(b)$.
        \end{enumerate}
        
        \item When $\lambda = \left(-\frac{1}{2},b|-\frac{1}{2}\right)$, we start with the case $b=-\frac{3}{2}$. We project $P_{-\frac{1}{2},-\frac{3}{2}|\frac{1}{2}} \otimes V$ onto the $\lambda$ block. We get that
        \begin{align*}
            \mathrm{pr}_{\lambda}\left(P_{-\frac{1}{2},-\frac{3}{2}|\frac{1}{2}} \otimes V\right)
            &= 2 \sum M_{-\frac{1}{2},-\frac{3}{2}|-\frac{1}{2}} + 2 \sum M_{\frac{3}{2},-\frac{1}{2}|-\frac{1}{2}} \\
            &+ \sum M_{\frac{1}{2},-\frac{3}{2}|\frac{1}{2}} + M_{\frac{3}{2},-\frac{3}{2}|\frac{3}{2}} + M_{\frac{3}{2},\frac{3}{2}|\frac{3}{2}}.
        \end{align*}
        First, by Lemma~\ref{lowest} and Proposition~\ref{filprop}, two copies of $P_{\lambda}$ must appear in the projection and the terms in $\sum M_{-\frac{1}{2},-\frac{3}{2}|-\frac{1}{2}}$ belong to $P_{\lambda}$. Now, the lowest remaining term is $M_{\frac{1}{2},-\frac{3}{2}|\frac{1}{2}}$, and since only one copy of it remains, it cannot belong to $P_{\lambda}$, which means $P_{\frac{1}{2},-\frac{3}{2}|\frac{1}{2}}$ must appear in the projection as a separate projective. Now, the only remaining terms are the two copies of $R(-\frac{3}{2})=\sum M_{\frac{3}{2},-\frac{1}{2}|-\frac{1}{2}}$, which could form $P_{\frac{3}{2},-\frac{1}{2}|-\frac{1}{2}}$. Thus, we again face two possibilities for $P_{\lambda}$, namely, $Q(-\frac{3}{2})=\sum M_{-\frac{1}{2},-\frac{3}{2}|-\frac{1}{2}}$ and $(Q+R)(-\frac{3}{2})$.
                
        Now, by projecting $Q(-\frac{3}{2}) \otimes V$ and $R(-\frac{3}{2}) \otimes V$ onto the $\left(-\frac{1}{2},-\frac{5}{2}|-\frac{1}{2}\right)$ block, we see that $P_{-\frac{1}{2},-\frac{5}{2}|\frac{1}{2}}$ also has two possible standard filtrations $Q(-\frac{5}{2})$ and $(Q+R)(-\frac{5}{2})$, defined similarly. In addition, $P_{-\frac{1}{2},-\frac{5}{2}|\frac{1}{2}} = Q(-\frac{5}{2})$ if and only if $P_{-\frac{1}{2},-\frac{3}{2}|\frac{1}{2}} = Q(-\frac{3}{2})$, as otherwise when we project the shorter projective onto the block of the longer, there would not be enough terms. The same argument carries as we induct on $b$. Thus, it remains to find the correct filtration for any specific value of $b$.
                
        We examine $\mu = \left(-\frac{1}{2},-\frac{5}{2}|-\frac{1}{2}\right)$ and show that $P_{\mu} = Q(-\frac{5}{2})$. Consider the projection $\mathrm{pr}_{\mu}\left(P_{-\frac{3}{2},-\frac{5}{2}|-\frac{3}{2}} \otimes S^2 V\right)$, which has $180$ terms. By applying Lemma~\ref{lowest}, we find that $P_{-\frac{5}{2},-\frac{5}{2}|-\frac{5}{2}}$ and four copies of $P_{-\frac{3}{2},-\frac{5}{2}|-\frac{3}{2}}$ must appear in the projection. Now, $60$ terms remain, and the lowest term is $M_{-\frac{1}{2},-\frac{5}{2}|-\frac{1}{2}}$, which appear $4$ times, which means that $P_{\mu}$ must appear four times. However, $(Q+R)(-\frac{5}{2})$ has $16$ terms and thus does not fit. Thus, 
        \[P_{-\frac{1}{2},b|-\frac{1}{2}} = Q(b) = \sum M_{-\frac{1}{2},b|-\frac{1}{2}}\]
        for all $b<-\frac{1}{2}$.
        It turns out that
        \begin{align*}
            \mathrm{pr}_{\mu}\left(P_{-\frac{3}{2},-\frac{5}{2}|-\frac{3}{2}} \otimes S^2 V\right) 
            &= P_{-\frac{5}{2},-\frac{5}{2}|-\frac{5}{2}} + 4P_{-\frac{3}{2},-\frac{5}{2}|-\frac{3}{2}} \\
            &+ 4P_{-\frac{1}{2},-\frac{5}{2}|-\frac{1}{2}} + P_{\frac{3}{2},-\frac{5}{2}|-\frac{3}{2}}.
        \end{align*}

        \item When $\lambda = (a,a-1|-a+1)$:
        \begin{enumerate}[label=(\theenumi.\arabic*), ref=\arabic*]
            \item If $a<-\frac{1}{2}$,
                \begin{align*}
                    \mathrm{pr}_{\lambda}\left(P_{a+1,a-1|-a+1} \otimes V\right) 
                    &= \sum M_{a,a-1|-a+1} + \sum M_{a+1,a|-a-1} \\
                    &+ \sum M_{a,a|-a} + \sum M_{-a,a|-a}.
                \end{align*}
                By Proposition~\ref{filprop}, all terms except for $\sum M_{-a,a|-a}$ must belong to $P_{\lambda}$. As no projective can form among the remaining two terms, the projection is equal to $P_{\lambda}$.
                
            \item If $a=-\frac{1}{2}$,
            \begin{align*}
                \mathrm{pr}_{\lambda}\left(P_{\frac{1}{2},-\frac{3}{2}|\frac{3}{2}} \otimes V\right) 
                &= \sum M_{-\frac{1}{2},-\frac{3}{2}|\frac{3}{2}} + M_{\frac{3}{2},-\frac{1}{2}|\frac{3}{2}} + M_{\frac{3}{2},\frac{1}{2}|\frac{3}{2}} \\
                &+ \sum M_{-\frac{1}{2},-\frac{1}{2}|\frac{1}{2}} + \sum M_{\frac{1}{2},-\frac{1}{2}|-\frac{1}{2}}.
            \end{align*}
            By Proposition~\ref{filprop}, $\sum M_{-\frac{1}{2},-\frac{3}{2}|\frac{3}{2}}$ and $\sum M_{-\frac{1}{2},-\frac{1}{2}|\frac{1}{2}}$ must belong to $P_{\lambda}$. As no projective can form among the remaining six terms, the projection is equal to $P_{\lambda}$.
        \end{enumerate}
            
        \item The proof for the case $\lambda = (a,a-1|a-1)$ follows similarly as the previous case. 
    \end{enumerate}
\end{proof}

\section{Jordan-H{\"o}lder Formulae for \texorpdfstring{$\osp(3|4)$}{osp(3|4)}}\label{sec7}
By BGG reciprocity, we can convert the standard filtration multiplicities for projective modules into Jordan-H{\"o}lder multiplicities of irreducible modules for Verma modules.

Let $\lambda \in X+\rho$ such that $\lambda-\rho$ is atypical, integral, and antidominant. Let $\mathcal{W}^{\lambda}$ be a minimal set of left-coset representatives of $\mathcal{W}/\mathcal{W}_{\lambda}$. Then, by applying the BGG reciprocity to Proposition~\ref{filprop}, we immediately get that the composition series of $M_{\sigma \lambda}$ ($\sigma \in \mathcal{W}^{\lambda}$) must include
\[
    \sum_{\tau \leq \sigma, \tau \in \mathcal{W}^{\lambda}} \left(L_{\tau \lambda} + L_{\tau \lambda - \alpha} + L_{\tau \lambda - \alpha - \beta}\right),
\]
where each term in the sum appears with multiplicity one only if it is linked to $\lambda$, and $\alpha,\beta \in {\Phi_{\overline{1}}}^+$ and $\mathrm{ht}(\alpha)>\mathrm{ht}(\beta)$.
For convenience, we denote this summation by
\[
    \sum L_{\sigma \lambda}.
\]
The following theorem follows from applying the BGG reciprocity to the character formulae we obtained in Theorems~\ref{thm:spo431} to \ref{thm:spo434}.
\begin{thm}\label{jordanholder}
    Let $\lambda-\rho = (a,b \ | \ c)-\rho$ be an atypical weight with $a,b,c \in \frac{1}{2} + \mathbb{Z}$ and $c \in \{\pm a, \pm b\}$. The Verma modules $M_{\lambda}$ of highest weight $\lambda-\rho$ have Jordan-H{\"o}lder formulae
    \[
        M_{\lambda} = \sum L_{\lambda}
    \]
    except in the following cases.
    \begin{enumerate}[label=(\arabic*)]
        \item Suppose that $\lambda-\rho=\left(a',b'|\frac{3}{2}\right)-\rho$ is atypical, and at least one of $a',b'$ is positive. Since at least one of $|a'|,|b'|$ is equal to $\frac{3}{2}$, suppose that $\{|a'|,|b'|\} = \{a, \frac{3}{2}\}$. Then, unless specified otherwise in cases below, $M_{\lambda}$ has the following composition series:
        \[
            M_{\lambda} = \sum L_{\lambda} + \sum_{*} L_{a,-\frac{1}{2}|\frac{1}{2}},
        \]
        where $\sum_{*} L_{a,-\frac{1}{2}|\frac{1}{2}}$ denotes the the sum of the terms in the set \[\{L_{-a,-\frac{1}{2}|\frac{1}{2}}, L_{a,-\frac{1}{2}|\frac{1}{2}}, L_{-\frac{1}{2},-a|\frac{1}{2}}, L_{-\frac{1}{2},a|\frac{1}{2}}\}\]
        that are lower than $L_{\lambda}$. The following subcases are exceptions to this case, which contain some additional terms than those given above.
        \begin{enumerate}[label=(\roman*)]
            \item When $\lambda=\left(\frac{3}{2},-\frac{1}{2}|\frac{3}{2}\right)$. We have
            \[
                M_{\lambda} = \sum L_{\lambda} + L_{-\frac{1}{2},-\frac{1}{2}|\frac{1}{2}} + \red{L_{-\frac{1}{2},-\frac{3}{2}|\frac{3}{2}}},
            \]
            where we use red to emphasize terms with multiplicity two (its first copy appears in $\sum L_{\lambda}$).
            
            \item When $\lambda=\left(\frac{3}{2},\frac{1}{2}|\frac{3}{2}\right)$. We have
            \[
                M_{\lambda} = \sum L_{\lambda} + L_{-\frac{1}{2},-\frac{1}{2}|\frac{1}{2}} + \red{L_{-\frac{1}{2},-\frac{3}{2}|\frac{3}{2}} + L_{\frac{1}{2},-\frac{1}{2}|\frac{1}{2}}}.
            \]
            
            \item When $\lambda=\left(\frac{3}{2},-\frac{3}{2}|\frac{3}{2}\right)$. We have 
            \begin{align*}
                M_{\lambda} 
                &= \sum L_{\lambda} + L_{-\frac{1}{2},-\frac{3}{2}|\frac{1}{2}} + L_{-\frac{3}{2},-\frac{1}{2}|\frac{1}{2}}\\
                &+ \red{L_{-\frac{3}{2},-\frac{5}{2}|\frac{5}{2}} + L_{-\frac{3}{2},-\frac{5}{2}|-\frac{5}{2}}}.
            \end{align*}
            
            \item When $\lambda=\left(\frac{3}{2},\frac{3}{2}|\frac{3}{2}\right)$. We have 
            \begin{align*}
                M_{\lambda} 
                &= \sum L_{\lambda} + \sum_{*} L_{\frac{3}{2},-\frac{1}{2}|\frac{1}{2}}\\
                &+ \red{L_{-\frac{3}{2},\frac{3}{2}|\frac{3}{2}} + L_{-\frac{3}{2},\frac{3}{2}|-\frac{3}{2}} + L_{-\frac{3}{2},-\frac{5}{2}|\frac{5}{2}} + L_{-\frac{3}{2},-\frac{5}{2}|-\frac{5}{2}}}.
            \end{align*}
        \end{enumerate}
        
        \item Suppose that $\lambda-\rho=\left(a',b'|\frac{1}{2}\right)-\rho$ is atypical, and at least one of $a'$, $b'$ is greater than $\frac{1}{2}$. 
        \begin{enumerate}[label=(\roman*)]
            \item When $\lambda=\left(-\frac{1}{2},b|\frac{1}{2}\right)$ or $\lambda=\left(\frac{1}{2},b|\frac{1}{2}\right)$ with $b>\frac{1}{2}$. We have 
            \[
                M_{\lambda} = \sum L_{\lambda} + \red{L_{-b,-\frac{1}{2}|\frac{1}{2}}}.
            \]
            
            \item Suppose that $\lambda=\left(a,-\frac{1}{2}|\frac{1}{2}\right)$ or $\lambda=\left(a,\frac{1}{2}|\frac{1}{2}\right)$ with $a>\frac{1}{2}$. We have 
            \[
                M_{\lambda} = \sum L_{\lambda} + \red{L_{-\frac{1}{2},-a|\frac{1}{2}} + L_{-a,-\frac{1}{2}|\frac{1}{2}}}.
            \]
        \end{enumerate}
        
        \item Suppose that $\lambda-\rho=(a,b \ | \ c)-\rho$ is atypical, and $a=|b|=|c|$.
        \begin{enumerate}[label=(\roman*)]
            \item When $\lambda=(a,-a|-a)$, we have
            \[
                M_{\lambda} = \sum L_{\lambda} + \red{L_{-a,-a-1|-a-1}}.
            \]
            \item When $\lambda=(a,-a|a)$ and $a\neq \frac{3}{2}$, we have
            \[
                M_{\lambda} = \sum L_{\lambda} + \red{L_{-a,-a-1|-a-1} + L_{-a,-a-1|a+1}}.
            \]
            \item The case $\lambda=\left(\frac{3}{2},-\frac{3}{2}|\frac{3}{2}\right)$ is given above.
            \item When $\lambda=(a,a|-a)$, we have
            \[
                M_{\lambda} = \sum L_{\lambda} + \red{L_{-a,a|-a} + L_{-a,-a-1|-a-1}}.
            \]
            \item When $\lambda=(a,-a|a)$ and $a>\frac{3}{2}$, we have
            \[
                M_{\lambda} = \sum L_{\lambda} + \red{L_{-a,a|-a} + L_{-a,a|a} + L_{-a,-a-1|-a-1} + L_{-a,-a-1|a+1}}.
            \]
            \item The case $\lambda=\left(\frac{3}{2},\frac{3}{2}|\frac{3}{2}\right)$ is given above.
            \item When $\lambda=\left(\frac{1}{2},\frac{1}{2}|\frac{1}{2}\right)$, we have
            \[
                M_{\lambda} = \sum L_{\lambda} + \red{L_{-\frac{1}{2},\frac{1}{2}|-\frac{1}{2}} + L_{-\frac{1}{2},-\frac{1}{2}|\frac{1}{2}} + L_{-\frac{1}{2},-\frac{3}{2}|-\frac{3}{2}} + L_{-\frac{1}{2},-\frac{3}{2}|\frac{3}{2}}}.
            \]
        \end{enumerate}
        \item When $\lambda=\left(\frac{5}{2},\frac{1}{2}|\frac{5}{2}\right)$, we have
        \[
            M_{\lambda} = \sum L_{\lambda} + L_{\frac{1}{2},-\frac{1}{2}|\frac{1}{2}}.
        \]
        \item When $\lambda=\left(\frac{5}{2},\frac{3}{2}|\frac{5}{2}\right)$, we have
        \[
            M_{\lambda} = \sum L_{\lambda} + L_{\frac{3}{2},-\frac{1}{2}|\frac{1}{2}}.
        \]

    \end{enumerate}
\end{thm}


\printbibliography

@article{beilinson1996koszul,
  title={Koszul duality patterns in representation theory},
  author={Beilinson, Alexander and Ginzburg, Victor and Soergel, Wolfgang},
  journal={Journal of the American Mathematical Society},
  volume={9},
  number={2},
  pages={473--527},
  year={1996}
}

@article{gorelik2002annihilation,
  title={Annihilation theorem and separation theorem for basic classical Lie superalgebras},
  author={Gorelik, Maria},
  journal={Journal of the American Mathematical Society},
  volume={15},
  number={1},
  pages={113--165},
  year={2002}
}

@article{bao2017kazhdan,
  title={Kazhdan-Lusztig theory of super type D and quantum symmetric pairs},
  author={Bao, Huanchen},
  journal={Representation Theory of the American Mathematical Society},
  volume={21},
  number={11},
  pages={247--276},
  year={2017}
}

@article{cheng2011super,
  title={Super duality and irreducible characters of ortho-symplectic Lie superalgebras},
  author={Cheng, Shun-Jen and Lam, Ngau and Wang, Weiqiang},
  journal={Inventiones mathematicae},
  volume={183},
  number={1},
  pages={189--224},
  year={2011},
  publisher={Springer}
}

@article{cheng2019character,
  title={Character formulae in Category $\mathcal{O}$ for exceptional Lie superalgebras D (2| 1; $\zeta$)},
  author={Cheng, Shun-Jen and Wang, Weiqiang},
  journal={Transformation Groups},
  volume={24},
  number={3},
  pages={781--821},
  year={2019},
  publisher={Springer}
}

@book{bao2018new,
  title={A New Approach to Kazhdan-Lusztig Theory of Type B Via Quantum Symmetric Pairs},
  author={Bao, H. and Wang, W.},
  isbn={9782856298893},
  url={https://books.google.com/books?id=rwfnvQEACAAJ},
  year={2018},
  publisher={Soci{\'e}t{\'e} Math{\'e}matique de France}
}

@article{brundan2003kazhdan,
  title={Kazhdan-Lusztig polynomials and character formulae for the Lie superalgebra $\mathfrak{gl}(m|n)$)},
  author={Brundan, Jonathan},
  journal={Journal of the American Mathematical Society},
  volume={16},
  number={1},
  pages={185--231},
  year={2003}
}

@article{Brundan_2016,
   title={Tensor Product Categorifications and the Super Kazhdan–Lusztig Conjecture},
   ISSN={1687-0247},
   url={http://dx.doi.org/10.1093/imrn/rnv388},
   DOI={10.1093/imrn/rnv388},
   journal={International Mathematics Research Notices},
   publisher={Oxford University Press (OUP)},
   author={Brundan, Jonathan and Losev, Ivan and Webster, Ben},
   year={2016},
   month={Sep},
   pages={rnv388}
}

@article{Cheng_2015,
   title={The Brundan–Kazhdan–Lusztig conjecture for general linear Lie superalgebras},
   volume={164},
   ISSN={0012-7094},
   url={http://dx.doi.org/10.1215/00127094-2881265},
   DOI={10.1215/00127094-2881265},
   number={4},
   journal={Duke Mathematical Journal},
   publisher={Duke University Press},
   author={Cheng, Shun-Jen and Lam, Ngau and Wang, Weiqiang},
   year={2015},
   month={Mar},
   pages={617–695}
}

@article{beilinsonbernstein81,
  title={Localisation de $\mathfrak{g}$-modules},
  author={Beilinson, Alexander and Bernstein, Joseph},
  journal={C. R. Acad. Sci., Paris, Ser},
  volume={292},
  number={15-18},
  year={1981},
}

@article{brylinski1981kazhdan,
  title={Kazhdan-Lusztig conjecture and holonomic systems},
  author={Brylinski, Jean-Luc and Kashiwara, Masaki},
  journal={Inventiones mathematicae},
  volume={64},
  number={3},
  pages={387--410},
  year={1981},
  publisher={Springer-Verlag}
}

@article{Lusztig1979,
author = {Lusztig, George and Kazhdan, David},
journal = {Inventiones mathematicae},
pages = {165-184},
title = {Representations of Coxeter Groups and Hecke Algebras.},
url = {http://eudml.org/doc/142660},
volume = {53},
year = {1979},
}

@article{gorelik10.2307/827094,
 ISSN = {08940347, 10886834},
 URL = {http://www.jstor.org/stable/827094},
 author = {Maria Gorelik},
 journal = {Journal of the American Mathematical Society},
 number = {1},
 pages = {167--184},
 publisher = {American Mathematical Society},
 title = {Strongly Typical Representations of the Basic Classical Lie Superalgebras},
 volume = {15},
 year = {2002}
}

@book{musson2012lie,
  title={Lie Superalgebras and Enveloping Algebras},
  author={Musson, I.M.},
  isbn={9780821868676},
  lccn={2011044064},
  series={Graduate studies in mathematics},
  url={https://books.google.com/books?id=\_FyDAwAAQBAJ},
  year={2012},
  publisher={American Mathematical Society}
}

@book{cheng2012dualities,
  title={Dualities and Representations of Lie Superalgebras},
  author={Cheng, S.J. and Wang, W.},
  isbn={9780821891186},
  lccn={2006936486},
  series={Graduate studies in mathematics},
  url={https://books.google.com/books?id=Mi11LAPjr08C},
  year={2012},
  publisher={American Mathematical Society}
}

@book{humphreys2008representations,
  title={Representations of Semisimple Lie Algebras in the BGG Category O},
  author={Humphreys, J.E.},
  isbn={9780821846780},
  lccn={2008012667},
  series={Graduate studies in mathematics},
  year={2008},
  publisher={American Mathematical Society}
}

@article{KANNAN2019231,
title = "Characters for projective modules in the BGG Category O for general linear Lie superalgebras",
journal = "Journal of Algebra",
volume = "532",
pages = "231 - 267",
year = "2019",
issn = "0021-8693",
doi = "https://doi.org/10.1016/j.jalgebra.2019.05.024",
url = "http://www.sciencedirect.com/science/article/pii/S0021869319302789",
author = "Arun S. Kannan"
}

@article{chengwangg3,
  title={Character formulae in category O for exceptional Lie superalgebra G (3)},
  author={Cheng, SJ and Wang, W},
  journal={arXiv preprint arXiv:1804.06951},
  year={2018}
}

\end{document}